\documentclass[12pt]{amsart}

\usepackage{amscd}

\usepackage{amsmath,amsfonts,amsthm,epsfig,latexsym,graphicx,amssymb}
\usepackage[english]{babel}
\usepackage{color}
\usepackage{verbatim}

\newtheorem{coro}{Corollary}[section]
\newtheorem{defi}{Definition}[section]
\newtheorem{prop}{Proposition}[section]

\newtheorem{theo}{Theorem}
\newtheorem{lemm}{Lemma}[section]
\newtheorem{ques}{Question}[section]

\newtheorem{exem}{Example}

\newtheorem{rema}{Remark}

\newtheorem{clai}{Claim}

\newtheorem{claim}{Claim}[section]

\def\R{I\kern -0.37 em R}
\def\N{I\kern -0.37 em N}
\def\Z{I\kern -0.37 em Z}

\def\supess_#1{\mathop{\rm supess}\limits_{#1}}
\def\infess_#1{\mathop{\rm infess}\limits_{#1}}

\def\DD{{\mathbb D}}

 \def\NN{{\mathbb N}} 
\def\PP{{\mathbb P}}
\def\QQ{{\mathbb Q}} \def\RR{{\mathbb R}} \def\SS{{\mathbb S}}
\def\TT{{\mathbb T}}

 \def\ZZ{{\mathbb Z}}

\def\Si{\Sigma}

\def\De{\Delta}

\def\Ga{\Gamma}

\def\cA{{\mathcal A}}  \def\cG{{\mathcal G}}  \def\cS{{\mathcal S}} 
  \def\cH{{\mathcal H}}   
\def\cC{{\mathcal C}}  \def\cI{{\mathcal I}} \def\cO{{\mathcal O}} 
\def\cD{{\mathcal D}}   \def\cP{{\mathcal P}} \def\cV{{\mathcal V}}
\def\cE{{\mathcal E}}    
\def\cF{{\mathcal F}}  \def\cL{{\mathcal L}} \def\cR{{\mathcal R}} \def\cX{{\mathcal X}}

\title[Circle at infinity of foliations of $\RR^2$]{\bf Action on the circle at infinity of foliations of $\RR^2$} 

\begin{document}



\author{Christian Bonatti}
\begin{abstract}This paper provides a canonical compactification of the plane $\RR^2$ by adding a circle at infinity  associated to  a  countable family of  singular foliations or laminations (under some hypotheses), generalizing an idea by Mather \cite{Ma}.  Moreover any homeomorphism of $\RR^2$ preserving the foliations extends on the circle at infinity.  

Then this paper provides conditions ensuring the minimality of the action on the circle at infinity induced by an action on $\RR^2$ preserving  one foliation or two transverse foliations. 

In particular the action on the circle at infinity associated to an Anosov flow $X$ on a closed $3$-manifold is minimal if and only if $X$ is  non-$\RR$-covered.

\end{abstract}

\maketitle
$$$$
\vskip 5mm
\footnotesize{

\textbf{Codes AMS: 37D20-37E10-37E35-37C86}}

\today


\section{Introduction}
\subsection{General presentation}
There are many ways to compactify the plane $\RR^2$, the simplest one being the Alexandrov compactification by point at infinity, and $\RR^2\cup\{\infty\}$ is the topological sphere $\SS^2$. This compactication is canonical and does not depend on any extra structure on $\RR^2$.  That is its strength, but also its weakness as it does not bring any informations on any structure we endow $\RR^2$. 

Another very natural and usual compactification of $\RR^2$ is by adding a circle at infinity, so that $\RR^2\cup \SS^1$ is the disc $\DD^2$. This compactification is not canonical: it consists in a homeomorphism $h\colon \RR^2\to \mathring{\DD^2}$, where $\mathring{\DD^2}$ is the open disc. Two homeomorphisms $h_1,h_2$ define the same compactification if $h_2\circ h_1^{-1}\colon \mathring{\DD^2}\to\mathring{\DD^2}$ extends on $\SS^1=\partial\DD^2$ as a homeomorphism of $\DD^2$.  There are uncountably many such a compactification.  

Here, we start be recalling Mather \cite{Ma} canonical compactification of the plane $\RR^2$, endowed with a foliation $\cF$,  by a circle at infinity $\SS^1_\cF$. Then we explore the flexibility of this contruction for extending it to more general objects. 
Thus, we provide an elementary (nothing sophisticated), simple (nothing too complicated), and unified construction which associates a compactification $\DD^2_\cF$ of the plane $\RR^2$ by the disc $\DD^2$ to  a countable family $\cF=\{\cF_i\}$ of foliations, non-singular or with singular points of saddle type,  which are pairwise transverse or at least have some kind of weak transversality condition at infinity, see the precise statements below. The boundary $\partial \DD^2_\cF$ is called \emph{the circle at infinity} of $\cF$ and is denoted by $\SS^1_\cF$. This compactification is unique, in the sense that the identity on $\RR^2$  extends as a homeomorphism on the circles at infinity of two such compactifications. 

For giving a concrete example, Corollary~\ref{e.algebraic} 
builds this canonical compactification $\DD^2_\cF$ associated to any countable family $\cF=\{\cF_i\}$  of singular foliations, where each $\cF_i$ is directed by a polynomial vector field on $\RR^2$ whose singular points are hyperbolic saddles.

The uniqueness of the compactification implies that any homeomorphism of $\RR^2$ preserving  $\cF$ (that is, permuting the $\cF_i$) extends as an homeomorphism of the compactification $\DD^2_\cF$ , inducing a homeomorphism of the circle at infinity $\SS^1_\cF$.

\subsection{Mather idea for building the circle at infinity}

The common setting for this unified construction are families of \emph{rays}, where a ray is a proper topological embedding of $[0,+\infty)$ on $\RR^2$. 
We require that the \emph{germs  of the rays} in the family are pairwize disjoint, meaning that the intersection between any two distinct rays is compact. The key idea is that a set of rays  in $\RR^2$ whose germs are pairwize disjoint  is  \emph{totally cyclically ordered}, and we will use this cyclic order for building the circle at infinity.

The key technical result (essentially due to \cite{Ma}) is:

\begin{theo}\label{t.rays} Let $\cR$ be a family of rays in $\RR^2$ whose germs are pairwise disjoint. Let $\cE\subset \cR$ be a countable subset which is \emph{separating for the cyclic order}, that is, any non-degenerate interval contains a point in $\cE$ (see Definition~\ref{d.separating}).

Then there is a compactification of $\RR^2$ by the disc $\DD^2$ so that:
\begin{itemize}
 \item any ray of $\cR$ tends to a point of the circle at infinity $\partial \DD^2=\SS^1$.
 \item any two distinct rays of $\cR$ tend to distinct points of $\SS^1$
 \item the points of $\SS^1$ which are the limit point of a ray in $\cR$ are dense in $\SS^1$.
\end{itemize}
Furthermore, this compactification is unique up to a homeomorphism of $\DD^2$ and does not depend on the separating countable set $\cE$.
\end{theo}

Then Theorem~\ref{t.union} provides such a canonical compactification for a countable union $\cR=\bigcup\cR_i, i\in I\subset \NN$ of families of rays, assuming that the germs of rays in $\cR$ are pairwise disjoint and each $\cR_i$ admits a countable separating subset $\cE_i$. The difficulty here is that $\cR$ by itself may not admit any separating family. The idea for solving this problem consists in considering a natural equivalence relation on $\cR$, identifying the rays which cannot be separated.   

\subsection{Countable families of transverse foliations}

A natural setting where we will apply this general construction  are (at most countable) families of transverse foliations on the plane $\RR^2$. Notice that any \emph{half leaf} of a (non-singular) foliation of $\RR^2$ is a ray.
An \emph{end of leaf} is the germ at infinity of an half leaf.
In this setting we get:
\begin{theo}\label{t.foliations} Let $\cF=\{\cF_i\}_{i\in I\subset \NN}$ be an at most countable family of  pairwise transverse foliations on the plane $\RR^2$.

There is a compactification $\DD^2_\cF\simeq \DD^2$ of $\RR^2$ by adding a circle $\SS^1_\cF=\partial\DD^2_\cF$ with the following properties:
 \begin{itemize}
  \item Any end of leaf tends to a point of the circle at infinity $\SS^1_\cF$,
  \item The set of ends of leaves tending to a same points  of $\SS^1_\cF$ is at most countable,
  \item For any non-empty open subset $O\subset \SS^1_\cF$ the set of   ends of leaves  having their limit in $O$ is uncountable.
 \end{itemize}

 This compactification  with these three properties is unique, up to a homeomorphism of $\DD^2_\cF$.
\end{theo}
The circle $\SS^1_\cF$ is called \emph{the circle at infinity} of the family $\cF=\{\cF_i\}_{i\in I\subset \NN}$.

\begin{rema} The countablity of the set of ends tending to the same point  implies that
\begin{itemize}\item the two ends of a given leaf always have distinct limits on $\SS^1_\cF$.
 \item if two leaves $L_1,L_2$ of the same  foliation $\cF_i$ have the same pair of limits of ends, they are equal (see Lemma~\ref{l.injective}).
\end{itemize}
 \end{rema}

Recall that foliations of $\RR^2$ may have leaves which are \emph{not separated} one from the other. The leaves which are separated from any other leaves are called \emph{regular leaves}. At most countably many leaves are not regular (see here Lemma~\ref{l.countable}). We will see that,

\begin{prop}Let $\cF=\{\cF_i\}_{i\in I\subset \NN}$ be an at most countable family of  pairwise transverse foliations on the plane $\RR^2$. Any two distinct  ends of regular leaves of the same foliation $\cF_i$  tend to two distinct points of $\SS^1_\cF$.
\end{prop}


 

Now, in the setting of Theorem~\ref{t.foliations} we can apply this theorem to each foliation $\cF_i$, $i\in I$ so that we get a family of compactifications $\DD^2_{\cF_i}$.  In fact, we get a compactification $\DD^2_J$ for any subfamily $J\subset I$ leading to an uncountable set of (maybe distinct) compactifications of $\RR^2$ by the disc $\DD^2$ (Example~\ref{e.different}  provides a simple example where these compactifications $\DD^2_J$, for  $J\subset I$, are pairwize distincts and uncountably many).  

These compactifications are easily related : for any subfamily $J\subset I$ the identity map on $\RR^2$ extends in a unique way by continuity as a projection $\Pi_{I,J}\colon \DD^2_{\cF}=\DD^2_I\to\DD^2_J$, which simply consists in colapsing the intervals in $\SS^1_I$ which do not contain any limit of an end of a leaf of a foliation $\cF_j, j\in J$.  

We will also see in a simple example that the assumption of \emph{at most countability} of the family $I$ of foliations cannot be erased: for instance, the conclusion Theorem~\ref{t.foliations} is false for the family of all afine foliations (by parallel straight lines) of $\RR^2$, parametrized by $\RR\PP^1$  (see Example~\ref{e.uncountable}).  

Example~\ref{e.Weierstrass} and Lemma~\ref{l.center-like} present a simple example where generic points (i.e. points in a residual set) of the circle at infinity $\SS^1_\cF$ of a foliation $\cF$ are not the limit of any  end of leaf of $\cF$. In this example, at the contrary, points in a dense subset of $\SS^1_\cF$ are limit of $2$ distinct ends of leaves. 

Lemma~\ref{l.finite} and~\ref{l.hyperbolic}  caracterize the points $p$ at the circle at infinity $\SS^1_\cF$, where $\cF$ is a foliation of $\RR^2$,  which are limit of several ends of leaves: the rays arriving at $p$ are ordered as an interval of $\ZZ$ and two successive ends bound  a hyperbolic sector.

Corollary~\ref{c.hyperbolics} generalizes Lemma~\ref{l.finite} and~\ref{l.hyperbolic} to the case of a countable family $\cF=\{\cF_i\}$ of transverse foliations and gives a complete description of the points in $\SS^1_\cF$ which are limit of several ends of leaves of the same $\cF_i$. 

\subsection{Countable families of non-transverse or singular foliations}

This construction can be generalized easily to the setting of families of non transverse or singular foliations. Let us present the most general setting we consider here. 

The foliations we consider admit singular points which are  \emph{saddle point with $k$-separatrices} (also called \emph{ $k$-prongs singularity}), $k>1$, the case $k=2$ corresponding to non-singular points. 

In this setting an \emph{end of leaf} is a ray of $\RR^2$ disjoint from the singular points and contained in a leaf. 

\begin{theo}\label{t.countable-singular}Let $\cF=\{\cF_i\}$, $i\in I\subset \NN$ be a family of singular foliations of $\RR^2$ whose singular points are each a saddle with $k$-separatrices with $k>2$. We assume that, given any two  ends $L_1, L_2$ of leaves  we have the following alternative:
\begin{itemize}
 \item either the germs of $L_1$ and $L_2$ are disjoints
 \item or the germs of $L_1$ and $L_2$ coincide. 
\end{itemize}
Then there is a compactification $\DD^2_\cF\simeq \DD^2$ of $\RR^2$ by adding a circle $\SS^1_\cF=\partial\DD^2_\cF$ with the following properties:
 \begin{itemize}
  \item Any end of leaf tends to a point of the circle at infinity $\SS^1_\cF$,
  \item The set of ends of leaves tending to a same points  of $\SS^1_\cF$ is at most countable,
  \item For any non-empty open subset $O\subset \SS^1_\cF$ the set of   ends of leaves  having their limit in $O$ is uncountable.
 \end{itemize}

 This compactification  with these three properties is unique, up to a homeomorphism of $\DD^2_\cF$.
\end{theo}

The hypothesis that the germs of ends of leaves are either equal or disjoint means that if the intersection of two leaves is not bounded, then these two leaves coincide on an half leaf. One easily checks that transverse foliations satisfy this hypothesis so that Theorem~\ref{t.foliations} is a straightforward corollary of Theorem~\ref{t.countable-singular}.

As a simple and natural example, we will see that any countable family $\cF=\{\cF_i\}$  of singular foliations, directed by  polynomial vector fields on $\RR^2$ whose singular points are hyperbolic saddles, satisfies the hypotheses of Theorem~\ref{t.countable-singular}:  this will prove  Corollary~\ref{e.algebraic} already mentioned above.

\subsection{Laminations}

The construction of the circle at infinity for foliations cannot be extended without hypotheses to the case of laminations, as leaves of laminations may fail to be lines, and can even be recurrent, see for instance example~\ref{e.plykin}. 

Theorems~\ref{t.lamination} and ~\ref{t.laminations} provide a generalisation of this construction to closed orientable laminations with no compact leaves and with uncountably many leaves.  This generalisation is not as satifactory as in the case of foliations, and we discuss some of the issues in Section~\ref{s.laminations}.
In particular Theorem~\ref{t.Lami} provides another canonical compactification, which holds also for countable oriented laminations with no compact leaves. 

\subsection{Minimality of the action on the circle at infinity}

Then we consider group actions $H\subset Homeo(\RR^2)$  on $\RR^2$ preserving $1$ or $2$ transverse foliations $\cF_i$.  The action of $H$ extends canonically on the circle at infinity and we will consider the following question: 
\begin{ques} Under what conditions on $H$ and on the foliations $\cF_i$ can we ensure that the action induced on $\SS^1_{\{\cF_i\}}$ is minimal?
\end{ques}

Our main result, for the case of $1$ foliation is the following:
\begin{theo}\label{t.mini} Let $\cF$ be a foliation of $\RR^2$ and $H\subset Homeo(\RR^2)$ be a group of homeomorphisms preserving $\cF$. We assume that for any leaf $L$, the union of its images $H(L)$ is dense in $\RR^2$. 

Then the two following properties are equivalent 
\begin{enumerate}
 \item the action induced by $H$ on the circle at infinity is minimal
 \item there are  pairs of distinct leaves $(L_1,L_2)$ and $(L_3,L_4)$ so that $L_1$ and $L_2$ are not separated from above and $L_3$ and $L_4$ are not separated from below. 
\end{enumerate}
\end{theo}

We will also generalize Theorem~\ref{t.mini} for families of transverse foliations. 

\subsection{Action on the circle at infinity of an Anosov flow}

Finally, we will consider the setting of an \emph{Anosov flow $X$} on a closed  $3$-manifold $M$. 
\begin{rema}In this setting it is known that $\pi_1(M)$ acts on $\SS^1$ by orientation preserving homeomorphisms, see Calegari Dunfield \cite{CaDu} inspirated by an unpublished work of Thurston \cite{Th}. This works follows completely distinct ideas that those presented here. 

Another construction of this circle at infinity (called \emph{ideal circle boundary}) is given in \cite{Fe4} for pseudo-Anosov flows. 
\end{rema}

Barbot and Fenley \cite{Ba1,Fe1} show that the lift  $\tilde X$ of $X$ is conjugated to the constant vector field $\frac{\partial}{\partial x}$ on $\RR^3$, so that the $\tilde X$-orbit space is a plane $\cP_X\simeq \RR^2$. 
This plane  $\cP_X$ is endowed with two transverse foliations $F^s, F^u$ which are  the projection of the stable and unstable foliations of $X$ lifted on $\RR^3$. Thus $(\cP_X,F^s,F^u)$ is the \emph{bifoliated plane} associated to $X$. Furthermore, the fundamental group $\pi_1(M)$ acts on $\cP_X$ and its action preserves both foliations $F^s$ and $F^u$. This action induces a natural action of  $\pi_1(M)$ on the circles at infinity $\SS^1_{F^s},\SS^1_{F^u},$ and $\SS^1_{F^s,F^u}$. 

A folklore conjecture asserts that two Anosov flows are orbitaly equivalent if and only if they induces the same action on the circle at infinity of $\{F^s,F^u\}$, see \cite{Ba1} for a result in this direction.  This conjecture as been recently announced to be proved in \cite{BFM}.

\cite{Ba1,Fe1} show that every leaf of $F^s$ is regular if and only if every leaf of $F^u$ is regular, and then the Anosov flow $X$ is called \emph{$\RR$-covered}. Our main result in that setting is 

\begin{theo}\label{t.Anosov} Let $X$ be an Anosov flow on a closed $3$-manifold and $(\cP_X,F^s,F^u)$ its bifoliated plane.  Let $\DD^2_{F^s,F^u}$, $\DD^2_{F^s}$, and $\DD^2_{F^u}$  be the compactifications associated to, respectively, the pair of foliations $F^s,F^u$, the foliation $F^s$ and the foliation $F^u$.  Then
\begin{enumerate}
 \item $\DD^2_{F^s,F^u}=\DD^2_{F^s}=\DD^2_{F^u}$ unless $X$ is orbitally equivalent to the suspension of an Anosov diffeomorphism of the torus $\TT^2$. 
 \item the action of $\pi_1(M)$ on the circles at infinity $\SS^1_{F^s,F^u}$,(or equvalently  $\SS^1_{F^s}$ or $\SS^1_{F^u}$) is   minimal if and only if $X$ is not $\RR$-covered. 
\end{enumerate}
\end{theo}
When $X$ is assumed to be transitive, this result is a simple consequence of Theorem~\ref{t.mini} above and a result by Fenley \cite{Fe3} ensuring that, assuming $X$ is  non-$\RR$-covered, then $F^s$ and $F^u$ admit non-separated leaves from above and non-separated leaves form below.  The proof of  Theorem~\ref{t.Anosov}, when $X$ is not assumed to be transitive, is certainly the most technically difficult argument of the paper, and is based on a description of hyperbolic basic sets for flows on $3$-manifolds.

Theorem~\ref{t.Anosov} implies that the minimality of the action on the circle at infinity is not related with the transitivity of the flow. However, according to \cite{BFM} the action on the circle at infinity charaterizes the dynamics of the flow.  This leads to the following question:

\begin{ques} What property of the action of $\pi_1(M)$ on the circle at infinity $\SS^1_{F^s,F^u}$ implies the transitivity of $X$? 

Can we find  the transverse tori by looking at the action of $\pi_1(M)$ on the circle at infinity? 
\end{ques}

\subsubsection{Aknowledgments} I would thank Sebastien Alvarez who invited me to present the results in this paper as a mini-course in Montevideo. This mini-course has been a motivation for ending this paper. I would also thanks 
Kathrin Mann for indicating me that the argument of Theorem~\ref{t.rays} is essentially contained in \cite{Ma}, and Michele Triestino for the statement and reference of Cantor-Bendixson theorem. 

\section{Circles at infinity for families of rays on the plane}
\subsection{Cyclic order}

Let $X$ be a set. A \emph{total cyclic order} on $X$ is a map  $\theta\colon X^3\to \{-1,0,+1\}$ with the following properties
\begin{itemize}
 \item $\theta(x,y,z)=0$ if and only if $x=y$ or $y=z$ or $x=z$.
 \item $\theta(x,y,z)=-\theta(y,x,z)=-\theta(x,z,y)$ for every $(x,y,z)$
 \item for every $x\in X$ the relation on $X\setminus\{x\}$ defined by 
 $$y<z\Leftrightarrow \theta(x,y,z)=+1$$ is a total order.
\end{itemize}

The emblematic example is:
\begin{exem} The oriented circle $\SS^1=\RR/\ZZ$  is totally cyclically ordered by the relation $\theta$ defined as follows: 
$\theta(x,y,z)=+1$ if and only if the $y$ belongs to the interior of the positively oriented simple arc staring at $x$ and ending at $z$. 
\end{exem}

If $\theta$ is a total cyclic order then for $x\neq z$ we define  the interval $(x,y)$ by 
$$(x,z)=\{y, \theta(x,y,z)=1\}.$$

We define the semi closed and closed  intervals $[x,z)$,$(x,z]$, and $[x,z]$ by adding the corresponding extremities $x$ or $z$ to the interval $(x,z)$. 

We say that $y$ is \emph{between} $x$ and $z$ is $y\in(x,z)$. 

The following notion of \emph{separating set} will be fundamental all along this work: 
\begin{defi}\label{d.separating}
Let $X$ be a set endowed with a total cyclic order.  A  subset $\cE\subset X$  is said \emph{separating} if  given any distinct $x,z\in X$ there is $y\in \cE$ (distinct from $x$ and $z$), between $x$ and $z$. 
\end{defi}

We will use the following easy exercize of topology of $\RR$ and $\SS^1$: 
\begin{prop}\label{p.cyclic-order} Let $X$ be a set endowed with a total cyclic order.  Assume that there is a countable subset $\cE\subset X$   which is separating.  

Then there is a bijection $\varphi$ of $X$ on a dense subset $Y\subset \SS^1$ which is strictly increasing for the cyclic orders of $X$ and of $\SS^1$.  Furthermore this bijection is unique up to a composition by a homeomorphism of $\SS^1$. 
\end{prop}The argument is classical but short and beautiful and I have no references for this precise statement. So let me present it: 
\begin{proof} One builds a bijection $\phi$ of $\cE$ to a contable dense subset $\cD\subset \SS^1$ by induction, as follows: one choose an indexation of $\cE=\{e_i,i\in \NN\}$ and of  $\cD=\{d_i,i\in\NN\}$.  One defines 
\begin{itemize}
 \item $\phi(e_0)=d_0$, $\phi(e_1)=d_1$  $i(0)=j(0)=0$ $i(1)=j(1)=1$
 \item consider $e_2$, it belongs either in $(e_0,e_1)$ or in $(e_1,e_0)$ and we chose $\phi(e_2)$ being $d_{j(2)}$ where $j(2)$ is the infimum of the $d_i$ in the corresponding interval $(d_0,d_1)$ or $(d_1,d_0)$.  One denotes i(2)=2. 
 \item consider now $j(3)=\inf \NN\setminus \{0,1,i(2)\}$ and define $\phi^{-1}(d_{j(3)})=e_{i(3)}$ where $i(3)$ is the infimum of the $i\notin\{0,1,2\}$ so that the position of $e_{i(3)}$ with respect to $e_0,e_1,e_2$ is the same as the position of $d_{j(3)}$ with repsect to $d_0,d_1,d_{j(2)}$. 
 \item \dots
 \item choose $i(2n)=\inf \NN\setminus\{i(k), k<2n\}$  and $\phi(e_{i(2n})$ is $d_{j(2n)}$ where $j(2n)$ is the infimum of the $j$ so that $d_j$ as the same position with respect to the $d_{j(k)}, k<2n$ as $e_{i(2n)}$ with respect to the $e_{i(k)}$. 
 \item choose $j(2n+1)=\inf \NN\setminus\{j(k), k<2n+1\}$  and $\phi^{-1}(d_{j(2n+1})$ is $e_{i(2n+1)}$ where $i(2n+1)$ is the infimum of the $i$ so that $e_i$ as the same position with respect to the $e_{i(k)}, k<2n+1$ as $d_{j(2n+1)}$ with respect to the $d_{j(k)}$. 
\end{itemize}
At each step of this construction one uses the separation property of $\cE$ and $\cD$ for ensuring the existence of the point announced in the same position. 

Once we built $\phi$ on $\cE$, it extends in a unique increasing way on $X$.  Then the separation property of $\cE$ implies that this extension is injective. 
\end{proof}

\begin{rema}\label{r.cyclic-order} Assume that $Z,\theta$ is a set endowed with a total cyclic order, and $\cE\subset X\subset Z$ are subsets so that $\cE$ is separating for
$X,\theta$. 

Let $\varphi\colon X\to Y$ be the map given by Proposition~\ref{p.cyclic-order}.  Then $\phi$ extends in a unique way as an (not strictly) increasing map $\Phi\colon Z\to \SS^1$: $\Phi(y)$ is between $\Phi(x)$ and $\Phi(z)$ only if $y$ is between $x$ and $z$. 

The non-injectivity of the  map $\Phi$ is determined as follows. Consider distinct points  $x\neq y$ of $Z$, then $\Phi(x)=\Phi(y)$ if and only if either $(x,y)$ or $(y,x)$ contains no more than $1$ element of $X$
\end{rema}

\subsection{Cyclic order on families of rays}
A \emph{line} is a proper embedding of $\RR$ in $\RR^2$.
A line $L$ cuts $\RR^2$ in two half plane. 
If $L$ is oriented, then there is an orientation preserving homeomorphism $h$ of $\RR^2$  mapping $L$ on the oriented $x$-axis of $\RR^2$ (endowed with the coordinates $(x,y)$).  This allows us to defined the upper and lower half-planes $\De^+_L$ and $\De^-_L$ as the pre-images by $h$ of $\{y\geq 0\}$ and $\{y\leq 0\}$ respectively. 

A  \emph{ray} is a proper embedding of $[0,+\infty)$ in $\RR^2$. Two rays define the same \emph{germ of ray} if their images coincide out of a compact ball. 
Two germs of rays are said disjoint if they admit disjoint realisations. 

\begin{exem}\begin{enumerate}
\item If $\cF$ is a foliation of $\RR^2$, every leaf defines to germs of rays called the \emph{ends of the leaf}.  By fixing an orientation of $\cF$ we will speak of the \emph{right and left ends} of a leaf. 
\item If $\{\cF_i\}_{i\in\cI}$ is a family of pairwise transverse foliations of $\RR^2$ then the set of all ends of leafs of the foliations $\cF_i$ is a family of pairwise disjoint germs of rays. 
\item Consider the set $\cS$ of all germs of rays $\gamma$ which are contained in an orbit of an affine (polynomial of degree $= 1$) vector field of saddle type.  Then $\cS$ is a family of pairwise disjoint germs of rays. 
            \end{enumerate}
 
\end{exem}

Next lemmas are  simple exercizes of plane topology: 
\begin{lemm}\label{l.3rayons} Let $\gamma_0,\gamma_1,\gamma_2$ be three disjoint rays. 

Assume that $C_1$ and $C_2$ are simple closed curves on the plane $\RR^2$ so that $\gamma_i\cap Cj$ is a unique point $p_{i,j}$, $i\in\{0,1,2\}, j\in\{1,2\}$. We endow $C_i$ with the boundary-orientation corresponding to the compact disc bounded by $C_i$.  Then the cyclic order of the $3$ points $p_{0,1},p_{1,1},p_{2,1}$ for the orientation of $C_1$  is the same as the cyclic order  of the $3$ points $p_{0,2},p_{1,2},p_{2,2}$ for the orientation of $C_2$.

We call it the cyclic order on the rays $\gamma_0,\gamma_1,\gamma_2$. 
\end{lemm}

\begin{lemm}\label{l.3germes} The cyclic order on three disjoint germs of rays $R_0,R_1,R_2$ does not depend on the choice of disjoint rays $\gamma_0,\gamma_1,\gamma_2$ realizing the germs $R_0,R_1,R_2$. 
\end{lemm}

\begin{coro}\label{c.3germes} Let $\gamma_0,\gamma_1,\gamma_2$ be three disjoint rays and $C$ be any simple close curve,  oriented as the boundary of the compact disc bounded by $C$, and having  a non-empty intersection with every $\gamma_i$. 

Let $p_i$ be the last point of $\gamma_i$ in $C$. Then the cyclic order of the $\gamma_i$ coincides with the cyclic order of the $p_i$ for the orientation of $C$. 
 
\end{coro}

\begin{coro}\label{c.entre} Let $R_0,R_1,R_2$ be three disjoint germs of rays. Let $L$ be an oriented  line whose  right end is $R_0$ and whose left end  is $R_2$. Then 
$R_1$ is between $R_0$ and $R_2$ for the cyclic order defined above (we denote $R_1\in(R_0,R_2)$)  if and only if it admits a realization contained in the upper 
half-plane $\De^+_L$ bounded by $L$. 
\end{coro}

Next proposition summerizes what we have got with this sequence of easy lemmas. 
\begin{prop} Consider $\cR$ a family of pairwise disjoint germs of rays. Then $\cR$ is totally cyclically ordered by the following relation: 

given three distinct germs of rays $R_0,R_1,R_2\in\cR $, 
the germ $R_2$ is between $R_1$ and $R_3$ if it admits a realisation contained in the upper-half plane $\De^+_L$  where $L$ an oriented line whose right end is $R_0$ and and whose  left end is $R_3$.  
\end{prop}

\subsection{Compactification of a family of rays by a circle at infinity}

In this paper \emph{a compactification of the plane  $\RR^2$ by the disc $\DD^2$} is by definition a homeomorphism between $\RR^2$ and the open disc $\mathring{\DD}^2$.

The aim of this section is the proof of Theorem~\ref{t.rays} which build a canonical compactification of $\RR^2$ associated to a family $\cR$ of rays, assuming it admits a countable separating (for the cyclic order) subset $\cE\subset \cR$. One of the main ingredients for the proof of Theorem~\ref{t.rays} is the following lemma which is an easy exercize of plane topology.


\begin{lemm}\label{l.cyclic-order} Let $\gamma_0,\gamma_1, \dots,\gamma_n$ be $n$ disjoint rays, $n>0$, and $K\subset \RR^2$ be a compact set.  Then there is a simple closed curve $C$ disjoint from $K$
bounding a compact disc $D$ containing $K$ in its interior and so that $C\cap \gamma_i$ consists in a unique point $p_i$, $i\in\{0,\dots,n\}$. 
\end{lemm}
\begin{proof} Just notice that there is a homeomorphism of $\RR^2$ mapping $\gamma_i$, $i\in\{1,\dots,n\}$ on radial (half-straight lines) rays. Then the proof is trivial. 
\end{proof}

\begin{proof}[sketch of proof of Theorem~\ref{t.rays}] We consider the set of rays endowed with the cyclic order and we embedd it in the circle $\SS^1$ by Proposition~\ref{p.cyclic-order}.  We denote by $E\subset \SS^1$ the dense countable subset corresponding to $\cE$.  
We define a topology on $\RR^2\coprod \SS^1$ by choosing a basis of neighborhood of the points in $\SS^1$ as the halph planes  bounded by lines $L$ whose both ends are rays $R_-, R_+$ in $\cE$ (each half plane correspond to a segment in 
$\SS^1\setminus \{R_-, R_+\}$. 

This topology does not depend of the choice of the countable separating subset $\cE$: if $\tilde \cE$ is another countable separating subset, each neighborhood of a point of $\SS^1$ obtain by using one family contains a neighborhood obtained by using the other family. 

Now one builds a map from $\RR^2$ on the interior of $\DD^2$ as follow: 
\begin{enumerate}
 \item  one considers the circles $C_n$, $n\geq 1$, of radius $\rho_n=1-\frac1{n+1}$ (that is , $C_n=\rho_n\cdot \SS^1$) endowed with the finite set of point 
$\rho_n\cdot x_1,\dots, \rho_n\cdot x_n$, where $E=\{x_n,n\geq 1\}$ is a choice of indexation of the countable set $E$. 
\item one choses by induction a realisation $R_n$ of the rays in $\cE$ 
and a family of simple closed loops $\gamma_n$ with the following properties:
\begin{itemize}
\item $\gamma_n$ is the boundary of a compact disc $D_n$ containing $D_{n-1}$ in its interior and containing the disk of radius $n$ of $\RR^2$.  In particular, $\bigcup_n D_n=\RR^2$. 
 \item $\gamma_n$ cuts the rays $R_m$, $m< n$ in a unique point. 
 \item one choses a representative of $R_n$ , disjoint from $R_m$, $m<n$, with origin on $\gamma_n$ and with no other intersection point with $\gamma_n$. 
 
\end{itemize}
Then, by definition of the cyclic order on the rays, the points $\gamma_n\cap R_i$, $i\leq n$, are cyclically ordered on $\gamma_n$ as the points $\rho_n\cdot x_1,\dots, \rho_n\cdot x_n$ on $C_n$
\item this allows us to choose a homeomorphism of $\RR^2$ to the interior of $\DD^2$ sending the loops $\gamma_n$ on the circles $C_n$ and the rays $R_n$ on the segments $[\rho_n,1)\cdot x_n$.  
\end{enumerate}
This homeomorphisms extends on the circle at infinity $\SS^1$ to $\partial \DD^2$.  
 
\end{proof}

\subsection{Union of countably many families of rays: the circle\label{ss.circle}}

\begin{prop}\label{p.union}
Let $\{X_i, i\in I\}$, $I\subset \NN$   be a finite or countable family of sets so that 
$\bigcup_i X_i$ is endowed with a total cyclic order. 
Assume that, for every $i$, there exist $E_i\subset X_i$ countable separating subset. 

On the union $X=\bigcup_i X_i$  we consider the relation 
$$x\sim y \Leftrightarrow ( [x,y]\cap E_i \mbox{ is finite for every } i \mbox{, or } [y,x]\cap E_i \mbox{ is finite for every } i).$$
In other words, $x\sim y$ if one of the two segments (for the cyclic order) bounded by $x$ and $y$ meets each family $E_i$ in at most finitely many points.

Then $\sim$ is an equivalence relation and every class contains at most $1$ point in each $X_i$. 

Let denote  
$$\pi\colon X\to \cX=\bigcup_i X_i/\sim.$$
We denote by $\cE$ the projection $\pi(E)$ of $E=\bigcup E_i$ on $\cX$. 

Then $\sim$ provides a complete cyclic order on $\cX$ and $\cE$ is a countable separating subset. 
\end{prop}
\begin{proof} The fact that $\sim$ is an equivalence relation is quite easy, as the union of two intervals meeting $X_i$ on finite sets meets $X_i$ on a finite set. 

Note that, assuming $x\sim y$,  the interval $[x,y]$ or $[y,x]$ (meeting every $E_i$ in finitely many points) is contained in the class of $x$ and $y$. Thus the class $[x]_\sim$ is a (proper) interval for the cyclic order. 

Consider $x,y\in X$ and assume that $[x,y]\cap E_j$ is finite for every $j$.  Assume that there is $i$ and distinct $z,t\in [x,y]\cap X_i$. 
Then the separating property of $E_i$ for $X_i$ ensures that $[x,y]\cap E_i$ is infinite contradicting the choice of the interval $[x,y]$. We deduces that every class meets every $X_i$ in at most $1$ point. 

Notice that this implies that the projection of $E_i$ on $\cX$ is injective. 

Consider $x,y,z\in X$ whose classes are distinct, and assume $z\in (x,y)$.  Consider now $a\sim x$, $b\sim y$ and $c\sim z$. Let $I_a,I_b,I_c$ be the  intervals $[x,a]$ or $[a,x]$, $[y,b] $ or $[b,y]$, $[z,c]$ or $[c,z]$ with finite intersections with the $E_i$, respectively.  Then these intervals are disjoints as there a contained in disjoint equivalence classes.  Thus the cyclic order for point in $I_a,I_b,I_c$ does not depend on the point in $I_a,I_b,I_c$ and thus $c\in(a,b)$. 

This shows that the quotient $\cX$ is endowed with a total cyclic order.

Consider now two distinct classes $[x]_\sim,[y]_\sim \in \cX$ of points $x,y\in X$. 
Thus 
there is $i$ so that $[x,y]\cap X_i$ is infinite 
Now the separating property of $E_i$ implies that $[x,y]\cap E_i$ 
is infinite.

As $\pi$ is injective on $E_i$ one gets that $([x]_\sim,[y]_\sim)\cap \pi(E_i)$ is infinite and thus $([x]_\sim,[y]_\sim)\cap \cE$ is infinite. 
One proved that $\cE$ is separating for $\cX$, ending the proof. 
\end{proof}

\subsection{Union of countably many families of rays: the compactification}

\begin{theo}\label{t.union} Let $\cR=\coprod_{i\in I}\cR_i$, $I\subset \NN$,  be a family of rays in $\RR^2$ whose germs are pairwise disjoint. Assume that for every $i\in I$ there is  a countable subset $E_i\subset \cR_i$ which is separating for $\cR_i$. 

Then, there is a compactification of $\RR^2$ by the disc $\DD^2$ so that:
\begin{itemize}
 \item any ray of $\cR$ tends to a point of the circle at infinity $\partial \DD^2=\SS^1$. 
 \item for every $i$, any two distinct rays of $\cR_i$ tend to distinct points of $\SS^1$
 \item for any non-empty open interval $J\subset \SS^1$ there is $i\in I$ so that at least $2$ rays in $\cR_i$ have there limit point in $J$. 
\end{itemize}
Furthermore, this compactification is unique up to a homeomorphism of $\DD^2$ and does not depend on the separating countable sets $E_i$.  
\end{theo}

Let us discuss item 3, whose formulation may be surprising. 
\begin{rema}\begin{enumerate}\item The third item implies that the points of $\SS^1$ which are the limit point of a ray in $\bigcup E_i$ are dense in $\SS^1$. 
\item  if $I$ is finite, item 3 is equivalent to the density of points in $\SS^1$ which are limit of rays. 
\item item 3 is necessary when there is an uncountable set of equivalence classe $[c]$, for the equivalence relation $\sim$ (defined at Section~\ref{ss.circle}),   which are infinite (necessarily countable) and contain a set $\cC([c])$ which is separating for the cyclic order.  In that case the uniqueness property announced in  Theorem~\ref{t.union} would be wrong if we replace item 3 by the density of points in $\SS^1$ which are limit of rays.  Example~\ref{e.separating-classes} provides a simple illustration of this trouble. 
\end{enumerate}
\end{rema}

\begin{proof}[Proof of Theorem~\ref{t.union}.] Let us denote $X=\cR$,  $X_i=\cR_i$, and  $E=\bigcup E_i$. 
Let $\sim$ be the equivalence relation defined in Proposition~\ref{p.union} on $X=\cR$ and let $\pi$ benote the projection $\pi\colon X\to \cX=X/\sim$.  
We choose a subset $Y\subset X$ with the following properties
\begin{itemize}
 \item each equivalence class $[\gamma]$ for $\sim$ contains exactly $1$ point $y_\gamma$ in $Y$
 \item if a class  for $\sim$ contains a point in $E$, then $y_\gamma\in E$. 
\end{itemize}
The existence of such subset $Y$ is certainly implied by the choice axiom  but this existence does not require this axiom.  For instance we can fix
\begin{itemize}
 \item if $[\gamma]_\sim\in  \pi(\bigcup \cE_i)$ then $y_\gamma$ is the unique point in $E_i$ in the $[\gamma]_\sim$ where $i$ is the smallest index for which $[\gamma]_\sim\in  \pi(\cE_i)$.
 \item if $[\gamma]_\sim\notin  \pi(E)$ then $y_\gamma$ is the unique point in $X_i\cap [\gamma]_\sim$ where $i$ is the smallest index for which $[\gamma]_\sim\cap X_i)\neq \emptyset$.
\end{itemize}

We denote $F=Y\cap E$. Notice that $\pi(F)=\pi(E)$ by construction.

Now  the projection $\pi\colon Y\to \cX$ is a bijection which is strictly increasing for the cyclic order.  As $\cE=\pi(E)$ is separating for $\cX$ (see Proposition~\ref{p.union}) one gets that $F$ is separating for $Y$. 

We can now apply Theorem~\ref{t.rays} to the set of rays $Y$ and the countable separating subset $F$. 
One gets a compactification of $\RR^2$ as a disc $\DD^2$ so that every ray in $Y$ tends to a point at the circle at infinity, two distinct rays in $Y$ tends to two distinct points, and the set of points at infinity limit of rays in $F$ is dense in the circle at infinity. 

Let us check now that every ray $\gamma\in \cR$ tends to a point at infinity. By construction of $Y$ there is $\sigma\in \cR$ so that $\sigma\sim \gamma$. We will prove that $\gamma$ tends to the limit point $s\in\SS^1$ of $\sigma$. For that we recall that a basis of neighborhood of $s$ is given by the half planes $\Delta_n$ bounded by lines $L_n$ whose both ends are $\sigma^-_n,\sigma^+_n\in Y$ so that $\sigma\in (\sigma^-_n,\sigma^+_n)$. 
Note that both intervals   $(\sigma,\sigma^+_n)$ and $(\sigma^-_n,\sigma)$ are infinite when one of the intervals $(\sigma,\gamma)$ or $(\gamma,\sigma)$ is finite.  One deduces that $\gamma\in (\sigma^-_n,\sigma^+_n)$ and thus the end of $\gamma$ is contained in $\Delta_n$.  Thus $\gamma$ tends to $s$. 

Now consider two distinct rays $\gamma, \gamma'\in \cR_i$. As every class for $\sim$ contains at most $1$ point of $\cR_i$ the classes of $\gamma$ and $\gamma'$ are distincts.  Thus there are $\sigma\neq \sigma'\in Y$ which are equivalent to $\gamma$ and $\gamma'$, respectively. The limit points of $\gamma$ and $\gamma'$ are those of $\sigma$ and $\sigma'$ respectively, which are distinct. We just checked that distincts rays in $\cR_i$ tends to distinct point at infinity.

\begin{clai}\label{c.class} Let  $\varphi\colon \RR^2\to\DD^2$ be a compactification satisfying the announced properties.  Then 
$2$ rays in $\cR$ tends to the same point of the circle at  infinity of the compactification if and only if they are equivalent for $\sim$
 
\end{clai}
\begin{proof} If two rays $a,b$ are not equivalent then each of  $(a,b)$ and $(b,a)$ contains  infinitely many rays in the same of the sets $\cR_i$, by definition of $\sim$.  As the limits of distinct rays in the same $R_i$ are different, one deduces that the limits of $a$ and $b$ are distinct. 

Conversely, if $a$ and $b$ have different limit points $r$ and $s$ in $\SS^1$  for the compactification then item 3 implies that there is $i\in I$ (resp. $j\in I$ so that , $(a,b)$ (resp. $(b,a)$) contains the ends of at least $2$ rays in $\cR_i$ (resp. $\cR_j$). As $\cR_i$ and $\cR_j$ admits separating subsets, this implies that  both $(a,b)\cap \cR_i$ and $(b,a)\cap \cR_j$ are infinite, so that $a\nsim b$.
\end{proof}

Consider now a non-empty open interval $J$ of the circle at infinity. We announced that there is $i$ for which  $J$ contains at least $2$ limits of rays in $\cR_i$. Recall that, according to Theorem~\ref{t.rays} the points in $J$ which are limit of rays  in $Y$ are dense in $J$. Thus there are at least $2$ points in $J$ which are limit of rays $R_1,R_2$ in $Y$.
This implies that, up to exchange $R_1$ and $R_2$ any ray in $\cR$ between $R_1$ and $R_2$ tend to a point in $J$. 
Now, the rays $R_1, R_2$ are not equivalent for $\sim$ according to Claim~\ref{c.class}.   By definition of $\sim$, there is $i$ so that there are infinitely rays between $R_1$ and $R_2$.  This proves Item 3 of Theorem~\ref{t.union}. 

Assume now that one has another compactification $\psi\colon\RR^2\to \DD^2$ satisfying also the announced properties. 
One deduces from Claim~\ref{c.class}  the fact that the images by $\psi$ of  two distinct rays in the set $Y$ (that we used for building the first compactification $\varphi$)  have two distinct limit points and that the limit points of the image by $\psi$ of rays in $Y$ are dense in $\SS^1$. Thus this new compactification satisfies the same property on the set of rays $Y$ as the one we built.  Now Theorem~\ref{t.rays} asserts that these compactifications differs from $\varphi$ by a homeomorphism of $\DD^2$, concluding the proof. 
\end{proof}

\begin{lemm}\label{l.union} Assume that $\cR$ satisfies the hypotheses of Theorem~\ref{p.union}, and let $\tilde \cR$ be a set of rays to that the germs of rays in $\cR\cup\tilde\cR$ are pairwise disjoint. 

Let $\RR^2\hookrightarrow \DD^2$ be a compactification given by Theorem~\ref{t.union} applied to $\cR$. Then any ray $\tilde \gamma$ in $\tilde \cR$ tends to a point at infinity.  
 
\end{lemm}
\begin{proof}The candidate for the limit  the intersection of all the closed intervals in $\SS^1$, bounded by limit of rays $a,b\in \cR$, so that $\tilde \gamma\in(a,b)$.  The basis of neighborhood of this point that we exhibit implies that indeed $\tilde \gamma$ tends to that point at infinity. 
 
\end{proof}

\subsection{An example with uncountably many compactifications}

The example below shows that, in the case of a infinite countable family $\cR=\{\cR_i\}, i\in\NN$, the compactification announced by Theorem~\ref{t.union} would not be unique if  we replace the item 3 of the conclusion by the density in $\SS^1_\cR$ of the limits of the rays in $\cR$. In the example below, 

\begin{exem}\label{e.separating-classes} Let $B\subset \RR^2$ be the open strip $\{(x,y)\in\RR^2, |x-y|<1\}.$
Let  $I$ be the set of linear lines with a rational inclination $\neq 1$. For any $i\in I$, let $\cF_i$ be the restriction to $B$ of the trivial  foliation by parallel straight lines directed by $i\in\RR\PP^1$.  For any $i$, let $\cR_i$ be the set of ends of  leaf of $\cF_i$.  Each $\cR_i$ admits a countable separating subset. Thus $\cR=\bigcup_{i\in I} \cR_i$ satisfies the hypotheses of Theorem~\ref{t.union}.  

Then there are uncountably many distinct compactifications of $\RR^2$ for which 
\begin{itemize}
\item any ray of $\cR$ tends to a point of the circle at infinity $\partial \DD^2=\SS^1$. 
 \item for every $i$, any two distinct rays of $\cR_i$ tend to distinct points of $\SS^1$
 \item the points of $\SS^1$ which are the limit point of a ray in $\bigcup E_i$ are dense in $\SS^1$. 
\end{itemize}
\end{exem}
 \begin{proof}
Let $\DD^2_\cR$  be the compactification of $B\simeq \RR^2$ by adding the circle at infinity $\SS^1_\cR$.

Every class $C$ for $\sim$ contains exactly $1$ ray in $\cR_i$ for any $i\in I$.  The rays in $C$ are ordered,for the cyclic order, as the points of $I$ in $\RR\PP_1$. So, $C$ is a separating set for itself. 

By construction of $\SS^1_\cR$, the class  $C$  corresponds to a point $c\in\SS^1_\RR$.  We can build another circle $\SS^1_{\cR,C}$ by opening the point $c$ in a segment $I_C$. Then, we can build a compactification $\DD^2_{\cR,C}$ so that the rays in $C$ tend to distinct points dense in $I_C$. In particular $I_C$ contains exactly $1$ limit point of a ray in $\cR_i$, for any $i$. 

We can repeat this argument opening not just a point in $\SS^1_\cR$ but a countable  subset $\cC$ of classes for $\sim$: we build a compactification  $\DD^2_{\cR,\cC}$ where the circle at infinity contains disjoint intervals $I_C, C\in\cC$, so that  each $I_C$ contains exactly $1$ limit point of a ray in $\cR_i$ for any $i$, ans these points are dense in $I_C$. 

As there are uncountably many such countable subsets $\cC$, this provides an uncountable family of pairwise distinct compactifications of $B$ satisfying  the $2$ first items and the density of the limit points of rays. 
\end{proof}

This shows that the uniqueness part of  Theorem~\ref{t.union} becomes wrong if we replace item 3 by the density in $\SS^1$ of the set of limits of rays.

\subsection{Uncountable families of families of rays}
 Theorem~\ref{t.union} is wrong for the union of an uncountable family of sets of rays, as shows the Example~\ref{e.uncountable} below.

\begin{exem}\label{e.uncountable} We consider $\RR^2$ endowed with all constant foliations $\cF_\theta, \theta\in\RR\PP^1$, where $\cF_\theta$ is the foliation whose leaves are the straight lines parallels to $\theta$. 

Then given any compactification of $\RR^2$ by $\DD^2$  for which every end of leaf tends to a point at infinity, then for all but a countable set of $\theta$ the right ends of the leaves of $\cF_\theta$ tends to the same point at infinity. 
 
\end{exem}
\begin{proof}The ends of leaves $F^+_\theta$ at the right, and those at the left $F^-_\theta$ of the foliation $ \cF_\theta$ are disjoint interval depending on the uncountable parameter
theta.  On the circle at most countably many disjoint intervals can be non trivial, ending the proof. 
 
\end{proof}

\subsection{Projection on the compactifications associated to each families}
Let us start with a very easy example, showing the at the circles at infinity associated to the subsets of a countable family of transverse foliations may lead to uncountabily many distinct compactifications, all quotient of the compactification associated to the whole family.  

\begin{exem}\label{e.different} Consider now a infinite countable subset $I\subset \RR\PP^1$ and consider the family $\cR_I$ of the leaves of the  constant foliations $\cF_\theta, \theta\in I$ on $\RR^2$ as already considered in example~\ref{e.uncountable}. 

Now the set of ends of leaves of each foliation $\cF_\theta$ corresponds to $2$ (because each leaf has $2$ ends) non-empty open intervals in $\SS^1_I$, and these intervals do not contain any end of leaf of any other foliation. 

Thus if $J,K\subset I$ are distinct subsets, the circles at infinity $\SS^1_{J}$ and $\SS^1_K$ are obtained by collapsing distinct intervals of $\SS^1_I$ and they are different. 

As the set $\cP(I)$ of all subset of $I$ in uncountable, this leads to an uncountable family of compatifications $\{\DD^2_J\}_{J\in\cP(I)}$ of 
$\RR^2$ by a circle at infinity. 

\end{exem}

This situation is quite general.

Let $\cR=\cR_1\coprod\cdots\coprod\cR_k$ be a family of rays in $\RR^2$ whose germs are pairwise disjoint. Assume that for every $i\in \{1,\dots,k\}$ there is  a countable subset $E_i\subset \cR_i$ which is separating.

Thus for every subset $I\subset\{1,\dots,k\}$, Theorem~\ref{t.union} provides a compactification $\DD^2_I$ of $\RR^2$, by the circle at infinity  corresponding to the rays in $\cR_i, i\in I$. 

\begin{prop} If $J\subset I$ then the identity map on $\RR^2$ extend by continuity as a projection $\Pi_{I,J}\colon \DD^2_I\to \DD^2_J$.  This projection consists in collapsing intervals of $\SS^1_I=\partial \DD^2_I$ which do not contain any limit points  of ray is $\cR_j, j\in J$. 

Furthermore if $K\subset J$ then 
$$\Pi_{I,K}=P_{J,K}\circ P_{I,J}.$$
 
\end{prop}

\begin{proof} We first define a projection $\pi_{I,J}\colon\SS^1_I\to \SS^1_J$  by using Remark~\ref{r.cyclic-order}: the subset $R_J \subset \SS^1_{I}$ of limit of rays of $\cR_J$
is is a strcitly increasing bijection with $\cR_J$.  Thus the increasing bijection of $\cR_J$ on a dense subset of $\SS^1_J$ induces an increasing bijection of $R_J$ in this dense subset of $\SS^1_J$. Now Remark~\ref{r.cyclic-order} asserts that this bijection extends on the whole $\SS^1_I$ in a not-stricly increasing map  $\pi_{I,J}\colon\SS^1_I\to \SS^1_J$. An increasing map with dense image is always continuous, so that $\pi_{I,J}$ is continuous. 

Finaly Remark~\ref{r.cyclic-order} asserts that the non-injectivity of $\pi_{I,J}$ consist in collapsing intervals of $\SS^1_I$ with at most $1$ point in $R_J$, which is the same topological operation as collapsing intervals 
with no points in $R_J$. 

For ending the proof, we will check that $\pi_{I,J}$ is the extension by continuity of the identity map of $\RR^2$ to the circles at infinity. 

Recall that we defined a basis of neighborhood of each point of the circle at infinity $\SS^1_I$ (resp. $\SS^1_J$ ) as the half-planes $\De^+_L$ bounded by lines whose both ends are in $\cR_I$ (resp. $\cR_J$). In particular, as $J\subset I$,  the neighborhoods of points at infinity
in $DD^2_J$ are still neighborhoods of points at infinity for $\DD^2_I$, proving that the map which is the identity from  $\RR^2=\mathring\DD^2_I$ to $\RR^2=\mathring\DD^2_J$
and is $\pi_{I,J}$ from $\SS^1_I$ to $\SS^1_J$ is continuous. This ends the proof. 
 
\end{proof}

\section{backgroung on foliations: regular leaves, non-separated leaves\label{s.background}}
\subsection{Non-singular foliations\label{ss.non-singular}}
Let $\cF$ be a foliation of $\RR^2$. Then 
\begin{enumerate}\item as $\RR^2$ is simply connected, $\cF$ is orientable and admits a transverse orientation.  Let us fix an orientation of $\cF$ and a transverse orientation. 
 \item every leaf is a line (i.e. a proper embedding of $\RR$ in $\RR^2$).
 \item a basis of  neighborhoods of a leaf $L$ is obtained by considering  the union of leaves through a transverse segment $\sigma$ through a point of $L$. 
 \end{enumerate}
 
 \begin{defi}\begin{itemize}\item two leaves $L_1$, $L_2$ are \emph{not separated} one from the other  if 
 they do not admit disjoint neighborhood.
 
  \item A leaf $L$ is called \emph{not separated} or \emph{not regular} if there is a leaf $L'$ which is not separated from $L$.  
  \item A leaf is called \emph{regular} if it is separated from any other leaf. 
  \end{itemize}
 \end{defi}

 We will need some times to be somewhat more specific.  
 
 Let $L_1$ and $L_2$ be distinct leaves of $\cF$.  Consider two segments $\sigma_i\colon [-1,1]$ transverse to $\cF$, positively oriented for the transverse orientation of $\cF$, and so that $\sigma_i(0)\in L_i$, $i=1,2$.   Then $L_1$ is not separated from $L_2$ means that there are sequences 
 $t^i_n$, $i=1,2$ tending to $0$ as $n\to+\infty$  so that $\sigma_1(t^1_n)$ and  $\sigma_2(t^2_n)$ belongs to the same leaf $L^n$. Then
 
\begin{itemize}
 \item as $\cF$ is transversely oriented and the $\sigma_i$ are positively oriented, one gets that 
 $t^1_n$   has the same sign as $t^2_n$, for every $n$. Futhermore all the $t^i_n$ have the same sign. 
 
 One says that $L_1$ and $L_2$ are \emph{not separated from above} (resp. \emph{from below}) if the $t^i_n$ are positive (resp. negative). 
 \item By shirinking the segments $\sigma_i$ if necessary one may assume that they are disjoint. Now, up to exchange $L_1$ with $L_2$ we may assume that $\sigma_1(t_n)$ is at the left of $\sigma_2(t_n)$ in the oriented leaf $L^n$. 
 We say that $L_1$ (resp. $L_2$) is not separated from $L_2$ at its  right (resp. at its left). 
\end{itemize}

Consider a leaf $L$ and $\sigma\colon[-1,1]\to \RR^2$ a transverse segment (positively oriented) with $\sigma(0)\in L$. Let $L_t$ be the leaf through $\sigma(t)$.  Let $U_t$, $t\in (0,1)$, be the closure of the connected component of $\RR^2\setminus (L_t\cup L_{-t})$ containing $L$.  Then 
\begin{lemm} \label{l.regular}
 The leaf $L$ is regular if and only if $$\bigcap_t U_t= L.$$
 
 The intersection $\bigcap_t U_t$ does not depend on the segment $\sigma$ and is denoted $\mathfrak{U}(L)$.

If $L$ is not regular, $\mathfrak{U}(L)$ as non-empty interior, and  the leaves which are not separated from $L$ are precisely the leaves in the boundary of $\mathfrak{U}(L)$. 
\end{lemm}
\begin{proof}A leaf $\tilde L$ not separated from $L$ is contained in every $U_t$ and is accumulated by leaves $L_{t_n}$ in the boundary of $U_{t_n}$. Thus $\tilde L$ is contained in the boundary of $\mathfrak{U}(L)=\bigcap_t U_t$. Furthermore one of the two half planes bounded by $\tilde L$ is contained in $U_t$ and therefore in $\mathfrak{U}(L)$. 

Conversely,$\bigcap_t U_t$ consist in entire leaves of $\cF$ and so does its boundary. Now any transverse segment through a leaf in the boundary of $\bigcap_t U_t$ crosses the boundary $L_t\cup L_{-t}$ of $U_t$ for $t$ small: that is the definition of being not sepatated from $L$. 
\end{proof}

 \begin{lemm}\label{l.countable} Let $\cF$ be a foliation of $\RR^2$. The set of not separated leaves is at most countable. 
 \end{lemm}
 
 \begin{proof} We consider a countable family of transverse lines $\Si$ whoses union cuts every leaf of $\cF$. It is enough to proof that such a tranverse line $\Si$ cuts at most a countable set of non-regular leaves $L$ admiting a non separated leaf $\tilde L$ from below.
 
 For that just notice that the $\mathfrak{U}(L)$ for $L\cap \Si\neq \emptyset$ are pairwise disjoint. Thus,there are at most countably many of them   with non-empty interior, ending the proof. 
 \end{proof}

Note that $L$ cuts the strip $U_t$, $t\in(0,1]$ in two strips 
$U^+_t$ and $U^-_t$ bounded respectively by $L_t\cup L$ and by $L_{-t}\cup L$, and we denote 
$$\mathfrak{U}^+(L)=\bigcap_t U^+_t\mbox{ and }\mathfrak{U}^-(L)=\bigcap_t U^-_t$$

Then 
\begin{lemm} $L$ is non-separated from above (resp. from below) if and only if $\mathfrak{U}^+(L)\neq L$ (resp. $\mathfrak{U}^-(L)\neq L$) and if and only if $\mathfrak{U}^+(L)$ (resp. $\mathfrak{U}^-(L)$) has non-emptyinterior. 
\end{lemm}

In the same spirit, $\sigma$ cuts the strip $U_t$ in two half strips $U^{left}_t$ and $U^{right}_t$ according to the orientation of $\cF$.  
Then one says that the \emph{right end} $L^+$ (resp. \emph{left end}$L^-$)\emph{ of $L$ is regular} if 
$$\mathfrak{U}_{right}=(L)\bigcap_t U^{right}_t= L^+\mbox{ (resp. }\mathfrak{U}_{left}(L)=\bigcap_t U^{left}_t=L^-\mbox{). }$$

We can be even more precise by considering the $4$ quadrants $U^{+,right}_t,U^{+,left}_t,U^{-,right}_t,U^{-,left}_t$ obtained by considering the intersections of $U^+_t$ and $U^-_t$ with $U^{right}_t$ and $U^{left}_t$. 
This allows us to speak on right or left ends of leaves non separated from above or from below, in the obvious way.

\subsection{Singular foliations: saddles with $k$-separatrices}
A singular foliation $\cF$ on $\RR^2$ is a foliation on $\RR^2\setminus Sing(\cF)$ where $Sing(\cF)$ is a closed subset of $\RR^2$. A \emph{leaf of $\cF$} is a leaf of the restriction of $\cF$ to $\RR^2\setminus Sing(\cF)$. Let us now recall the notion of saddles with $k$-separatrices, also called $k$-prongs singularities. 

We denote by $A_0$ the quotient of $[-1,1]^2$ by the involution $(x,y)\mapsto (-x,-y)$. The projection of $(0,0)$ on $A_0$ is still called $0,0$.  Note that the horizontal foliation (whose leaves are the segments $[-1,1]\times \{t\}$ is invariant by $(x,y)\mapsto (-x,-y)$, and therefore passes to the quotient on $A_0\setminus (0,0)$ and we denote by $\cH_1$ the induced foliation on $A_0\setminus\{(0,0)\}$.

A \emph{$1$-prong singular point $p$} of $\cF$ is a point of $Sing(\cF)$ which admits a neighborhood $U$ and a homeomorphism $h$ from $U$ to $A_0$ so that $h(p)=(0,0)$ and $h$ maps $\cF$ on $\cH_1$.

We denote by $A_k,\cH_k$ the cyclic ramified cover of $A_0$ at the point $(0,0)$ with $k$ leaves, endowed with the lift of $\cH_1$. 

A \emph{$k$-prongs singular point $p$},  equivalently a \emph{sadlle point with $k$ separatrices} of $\cF$ is a singular point admiting a homeomorphism of a neighborhood onto $A_k$ mapping $p$ on $(0,0)$ and $\cF$ on $\cH_k$. A separatrix of the saddle point $p$ is the leaf of $\cF$ containing a  connected  component of the lift of $]0,1]\times \{0\}$.

\begin{rema} \begin{itemize}\item If $p$ is a $2$-prongs singular point of $\cF$, then the foliation $\cF$ can be extended on $p$ so that $p$ is not singular. 

\item The Poincar\'e-Hopf index of a $k$-prongs singular point is $1-\frac k2$. 
\end{itemize}

\end{rema}

A \emph{foliation with singularities of saddle type} on $\RR^2$ is a singular foliation  for which each singular point is a saddle with $k$ separatrices, $k>2$.

\subsection{Leaves of singular foliations}

\begin{lemm}\label{l.PH} Let $\cF$ be a foliation on $\RR^2$ with singular points of saddle type.  Let $\sigma\colon [0,1]\to \RR^2\setminus Sing(\cF)$ be a segment transverse to $\cF$. Then for every leaf $\gamma$ one has
$$\# \sigma\cap\gamma\leq 1,$$
where $\#$ denotes the cardinal. 
 \end{lemm}
 \begin{proof}Assume (arguing by contradiction) that $\sigma\cap\gamma\geq 2$. Let $x,y$ be two successive (for the parametrisation of $\gamma$) intersection points with $\sigma$. The concatenation of the segments $[x,y]_\gamma$ and $[y,x]_\sigma$ is a simple closed curve $c$ in $\RR^2\setminus Sing(X)$. By Jordan theorem $c$ bounds a disc $D$ in $\RR^2$ and the Poincaré Hopf index of $\cF$ on $D$ is either equal to $1$, if $\gamma$ cuts $\sigma$ with the same orientation at $x$ and $y$, or $\frac 12$ otherwise: anyway this index is strictly positive.  However, this index is the sum of the Poincaré Hopf index of the singular points of $\cF$ contained in $D$. As each of them is negative, that is a contradiction, ending the proof.  
 \end{proof}
 
 The same argument shows that 
 \begin{lemm}\label{l.PH2} Let $\cF$ be a foliation on $\RR^2$ with singular points of saddle type.  Then $\cF$ has no compact leaves
 \end{lemm}
 \begin{proof}The index of $\cF$ on the disc bounded by a compact leaf whould be $1$ which is impossible with singular points with negative index. 
 \end{proof}

\begin{coro}\label{c.PH} Let $\cF$ be a singular foliation of $\RR^2$ whose singular points are all saddle points with at least $3$ separatrices. 
Then every half leaf of $\cF$ is either a ray or tends to a singular point $p$ of $\cF$ and is contained in a separatrix  of $p$. 
\end{coro}
\begin{proof} Consider the Alexandrov compactification of $\RR^2$ by a point at infinity.  Consider a leaf $\gamma$ and choose a parametrisation $\gamma(t)$. 
Consider  $$\limsup_{t\to+\infty} \gamma(t)=\bigcap_{t>0}\overline{\gamma([t,+\infty)},$$
where the closure is considered in $\RR^2\cup\{\infty\}$. 
It is a decreasing intersection of connected compact sets, and hence it is a non-empty connected compact set. 

If $\limsup_{t\to+\infty} \gamma(t)$ is not just a point, if contains a regular point $x$ of $\cF$, hence it cuts infinitely many times any transverse segment through $x$, which is forbidden by Lemma~\ref{l.PH}. 

Now $\limsup_{t\to+\infty} \gamma(t)$ is either the point $\infty$ or is a singular point of $\cF$, which is the announced alternative. 
\end{proof}

\subsection{Regular leaves of singular foliations}

Let $\cF$ be a foliation with  singular points of saddle type, $L_0$ a leaf of $\cF$ and $\sigma$ be a transverse segment through the point $\sigma(0)\in L_0$. 

The set of $t$ so that $\sigma(t)$ is contained in a separatrix of a singular point is at most countable. For any $t$ so that $\sigma(t)$ and $\sigma(-t)$ are not in a separatrix of  a singular point, the leaves $L_t$  and $L_{-t}$ through  $\sigma(t)$ and $\sigma(-t)$ are disjoint lines and therefore cut $\RR^2$ in $3$ connected components. We denote by $U_t$ the closure of the connected component of $\RR^2\setminus (L_t\cup L_{-t})$ containing $L_0$.  Notice that $U_t$ is a strip (homeomorphic to $\RR\times [-1,1]$)  bounded by $L_t\cup L_{-t}$ and saturated for $\cF$. 

\begin{lemm}With the notation above  $\bigcap_t U_t$ is a non-empty  closed subset of $\RR^2$ saturated for $\cF$ and we have the following alternative:
\begin{itemize}
 \item either $\bigcap_t U_t=L_0$ and $L_0$ is a non-singular leaf of $\cF$, 
 \item or $\bigcap_t U_t$ has non-empty interior. 
\end{itemize}

Furthermore, $\bigcap_t U_t$ does not depend on the choice of the transverse segment $\sigma$  through $L_0$ and is denoted $\mathfrak{U}(L_0)$.  
\end{lemm}
\begin{proof} $\mathfrak{U}(L_0)$ is saturated for $\cF$. If it contains a non-singular leaf, it contains one of the half planes bounded by this leaf. If it contains a singular leaf, it contains the corresponding singular point, and then  it contains at least one of the sectors bounded by the separatrices. 
\end{proof}

\begin{defi}With the notation above, the leaf $L_0$ is called \emph{regular} if $\mathfrak{U}(L_0)=L_0$, and will be called \emph{non-regular} otherwise. 
\end{defi}

\begin{rema}If $L_0$ is a separatrix of a singular point, then it is non-regular. 
\end{rema}

As in the case of non-singular foliations we have:  
\begin{prop}Let $\cF$ be a foliation with singular points of saddle type.  Then the set of non-regular leaves is at most countable. 
\end{prop}
\begin{proof} For any transverse segment $\sigma$ let denote by $L_t$ the leaf through $\sigma_t$.  Then by construction the closed sets $\mathfrak{U}(L_t)$ are pairwise disjoint. Thus at most countably many of them may have non-empty interior, that is, at most countably many of leaves $L_t$ are non-regular.  We conclude the proof by noticing that $\cF$ admits a countable family of transverse segment $\sigma_n, n\in \NN$ every leaf of $\cF$ cuts at least $1$ segment $\sigma_n$.  
\end{proof}

The leaves of a foliations have two ends, and the notion of regular leaves can be made more precise, looking at each of its ends. 

More precisely, let $L_{0,+}$ be an half leaf of $\cF$, and let $\sigma$ be a transverse segment so that $\sigma(0)$ is the initial point of $L_{0,+}$. For any $t$ so that $\sigma(t)$ and $\sigma(-t)$ do not belong to a separatrix of a singular point, we consider $L_{t,+}$ and  $L_{t,-}$ the half leaves starting at $\sigma(t)$ and $\sigma(-t)$ in the same side of $\sigma$ as $L_{0,+}$. We denote by $U_{t}(L_{0,+})\subset \RR^2$  the closed half plane containing $L_{0,+}$ and bounded by  the line of $\RR^2$ obtained by  concatenation of $L_{t,+}$ , $\sigma([-t,t])$ and $L_{t,-}$.   We denote $\mathfrak{U}(L_{0,+})=\bigcap_t U_{t}(L_{0,+})$.  Then :
\begin{itemize}
 \item either $\mathfrak{U}(L_{0,+})= L_{0,+}$ and one says that the half leaf  $L_{0,+}$ (or equivalently, the end of $L_0$ corresponding to $L_{0,+}$) is regular
 \item or $\mathfrak{U}(L_{0,t})\neq L_{0,t}$ is a closed subset  with non-empty interior. 
\end{itemize}

A leaf is regular if and only if its two ends are regular, and the set of non-regular ends of leaves is at most countable.

\subsection{Orientations}

A foliation with  singular points of saddle type  is locally orientable (and transversely orientable) in a neighborhood of a singular point $x$ if and only if the number of separatrices of $x$ is even. 

Thus a foliation of $\RR^2$ whose singular points are sadlles with even numbers of separatrices is locally orientable and transversely orientable, and therefore is globally orientable and transversely orientable, as $\RR^2$ is simply connected. 

Let $\cF$ be a foliation with  singular points of saddle type with even numbers of separatrices, and fix an orientation and transverse orientation of $\cF$. 

Thus every  leaf $L$ have a right and left end.  We defined $\mathfrak{U}^{right}(L)$ and $\mathfrak{U}^left(L)$ so that $L$ can be regular at the right or at the left. 

If $L_0$ is a leaf which is not a separatrix and $\sigma$ be a transverse segment with $\sigma(0)\in L_0$. 
One defines in the same way  the notions of being  regular from above and from below,  for $L_0$ or for each of its two ends. 

For instance $L^{right}_0$ is regular from above if $\mathfrak{U}_+(L_0^{right})=\bigcap_t U_{t,+}(L_0^{right})=L_0^{right}$ where $U_{t,+}(L^{right})$ is bounded by $L_0^{right}$, $\sigma([0,t])$ and $L^{right}_t$.

\section{The circle at infinity of a family of foliations}\label{s.feuilletage}
\subsection{The circle at infinity of a foliation of $\RR^2$: statement}
 
 The aim of this section is to recall the following  result essentially due to \cite{Ma} and to present a short proof of it.  
 
 \begin{theo}\label{t.feuilletage} Let $\cF$ be a foliation of the plane $\RR^2$, possibly with singularities of saddle type. Then there is a compactification $\DD^2_\cF\simeq $ of $\RR^2$  by  adding a circle at infinity $\SS^1_\cF=\partial\DD^2_\cF$ with the following property: 
 \begin{itemize}
 \item any half leaf tends either to a saddle point or to a point at infinity. 
  \item given a point $\theta\in \SS^1_\cF$ the set of ends of leaves tending to $\theta$ is at most countable. 
  \item the subset of $\SS^1_\cF$ corresponding to limits of regular ends of  leaves is dense in $\SS^1_\cF$. 
 \end{itemize}
 
 Furthermore this compactification of $\RR^2$ by $\DD^2$ with these three properties is unique, up to a homeomorphisms of the disk $\DD^2$. 
 \end{theo}
 \begin{rema} If $L^+_1\neq L^+_2$ are  two ends of leaves tending to the same point $\theta\in\SS^1_\cF$, then  
  $L_2\subset \mathfrak{U}^+$. In particular, the ends $L^+_1$ and $ L^+_2$ are not regular. 
 \end{rema}

\begin{coro}\label{c.extension} If a homeomorphisms $f$ of the plane $\RR^2$ preserves the foliation $\cF$ then it extends in a unique way as a homeomorphism $F$ of  the compactification 
$\DD^2_\cF$. 

Furhtermore the restriction of $F$ to $\SS^1_\cF$ is the identity map if and only if $f$ preserves every leaf of $\cF$ and preserves the orientation on each leaf. 
\end{coro}
\begin{proof}The first part is, as already noted, a straightforward consequence of the uniqueness of the compactification.

If $f$ preserves every leaf and preserves the orientation of the leaves, then it preserves every end of leaf. Thus the extension $F$ fixes every point of $\SS^1_\cF$ which is limit of an end of leaf.  As the limit points of end of leaves are dense in $\SS^1_\cF$ one deduces that the restriction of $F$ to $\SS^1_\cF$ is the identity map. 

Conversely, assume that $F$ is the identity on $\SS^1_\cF$.  If $\theta\in \SS^1_\cF$ is the limit of an unique end of leaf $L_+$ then $L_+$ is preserved by $f$. 

Thus $f$ preserves every regular end of leaf. As the regular leaves are dense in $\RR^2$, one deduces that $f$ preserves every oriented leaf, concluding.

\end{proof}


\subsection{Proof of Theorem~\ref{t.feuilletage}}

We denote by $Reg(\cF)$ the set of regular leaves of $\cF$ and by 
$\cR(\cF)$ the set of ends of regular leaves (any non singular leaf and in particular any regular leaf has two ends). Recall that $\cR(\cF)$ is a family of disjoint rays of $\RR^2$ and therefore is cyclically ordered. 

\begin{lemm}\label{l.separating}
  If $D$ is a family of  regular leaves whose union in dense in $\RR^2$, then the set $\cD$ of ends of the leaves in $D$ is a separating family for the set of ends of regular leaves $\cR(\cF)$. 
\end{lemm}
\begin{proof} Let $L_0$ be a regular leaf of $\cF$, $\sigma\colon [-1,1]\to \RR^2$ a segement transverse to $\cF$ with $\sigma(0)\in L_0$ and $U_t$ the family of neighborhoods of $L_0$ associated to the transverse segment $\sigma$. Our assumption implies that for a dense subset of $t\in [-1,1]$, the leaf $L_t$ belongs to $\cD$. Consider a sequence $t_n\in[-1,1], n\in \ZZ$ so that
\begin{itemize}\item $L_n=L_{t_n}\in \cD$
 \item $t_n\to 0$ as $|n|\to\infty$
 \item $t_n$ as the same signe as $n\in\ZZ$
\end{itemize}

Let $L_n^+$ and $L_n^-$ be the half leaves of $L_n$ (for the orientation given by the transverse orientation induced by $\sigma$). As $L_0$ is regular one gets 
$\mathfrak{U}(L_0^+)=L_0^+.$  This implies that $L_0^+$  (resp. $L_0^-$)  is the intersection of the intervals (for the cyclic order) $[L_{-n}^+, L_{n}^+]$ (resp.$[L_{n}^-, L_{-n}^-]$) for $n>0$. In other words,   the rays $L_{-n}^+, L_{n}^+$
(resp. $L_{n}^-, L_{-n}^-$) are separating the ray $L_0^+$  (resp. $L_0^-$)  from any other ray in $\cR(\cF)$ (and indeed from any other ray of leaf, regular or not ), concluding the proof. 

\end{proof}

We are now ready to prove Theorem~\ref{t.feuilletage}

\begin{proof}[Proof of Theorem~\ref{t.feuilletage}]
We chose a countable set $E$ of regular leaves whose union is dense in $\RR^2$. According to Lemma~\ref{l.separating} the set $\cE$ of ends of leaves in $E$ is a countable separating subset of $\cR(\cF)$.  Thus we may apply Theorem~\ref{t.rays}. 

One gets a compactification of $\RR^2$ by the disc $\DD^2_\cF\simeq \DD^2$, so that every two distinct ends of regular leaves tend to two distinct points at the circle at infinity $\SS^1_\cF$ and these points are dense on the circle and this compactification does not depend of the choice of the family. This prove the items 2 and 3 of the theorem, and also proves that these two items are enough for the uniqueness of this compactification. 

It remains to prove the first item, that is to show that the rays contained in  non-regular leaves also tend to points on $\SS^1_\cF$. That is done by Lemma~\ref{l.union}. 
\end{proof}

\begin{rema}\label{r.line} Let $\cF$ be a foliation (possibly with saddles).  Then every 
line $L$ transverse to $\cF$ has $2$ distinct limit points at infinity corresponding to its $2$ ends. 
 
\end{rema}
\begin{proof} The two ends of $L$ are rays disjoint from the ends in $\cR(\cF)$ (that is of the ends of leaves of $\cF$), as any transverse segment intersects any leaf in at most $1$ point.  Now Lemma~\ref{l.union} implies that the ends of $L$ tends to points on $\SS^1_\cF$.  These points are distinct because the regular half leaves through $L$ are between these two ends. 
\end{proof}

\begin{lemm}\label{l.injective} Let $\cF$ be a foliation (possibly with saddles). Given any two (non-singular) leaves $L_1,L_2$, if the ends of $L_1$ and $L_2$ tend to the same $2$ points in $\SS^1_\cF$ then $L_1=L_2$. 
\end{lemm}
\begin{proof} Assume $L_1\neq L_2$ share the same end points. Then the leaves in the strip bounded by $L_1\cup L_2$ would have their ends on the same points in $\SS^1_\cF$ contradicting the fact that at most contably many ends of leaves share the same end point on $\SS^1_\cF$. 
\end{proof}

As a by-product of the proof of Lemma~\ref{l.separating} we get the following:
\begin{lemm}\label{l.continuity}Let $\cF$ be a foliation (maybe with saddle-like singular points) and le $\sigma\colon[-1,1]\to \RR^2\setminus Sing(\cF)$ be a transverse segment. Let $\{L^+_t\}$ and $\{L^-_t\}$ be the half leaves starting at $\sigma(t)$.  Consider the map associating to $t\in(-1,1)$ the limit point of $L^+_t$ on $\SS^1_\cF$.  Then $t$ is a continuous point of this map if and only if $L^+_t$ is a regular end. 
\end{lemm}

\subsection{Points at $\SS^1_\cF$ limit of several ends of leaves: hyperbolic sectors} 
\begin{lemm}\label{l.finite} Let $A$ and $B$ be distinct  ends of leaves.  Then the following properties are equivalent
\begin{itemize}
 \item There are no end of regular leaf between $A$ and $B$.
 \item The set of ends of leaves between $A$ and $B$ is at most countable.
 \item The set of ends of leaves between $A$ and $B$ is finite.
\end{itemize}
\end{lemm}
\begin{proof}First assume that there is an end $L^+$ of a regular leaf $L$ between $A$ and $B$. We will prove that the interval $(A,B)$ is uncountable. 

Consider the neighborhood $U_t$ of $L$ associated to a transverse segment $\sigma$ with $\sigma(0)\in L$.  As $L$ is regular, one gets that $\mathfrak{U}(L)=\bigcap_t U_t=L$. As a consequence there is $t$ so that $A$ and $B$ are out of $U_t$. 

First assume that $A$ and $B$ are in the same connected component of $\RR^2\setminus U_t$. Then there is a line $L$ whose left end is $B$ and whose right end is $A$ and which is disjoint from $U_t$. One deduces that one of the interval $(A,B)$ and $(B,A)$ contains no end of leaf in $U_t$ (this cannot be $(A,B)$ which contains $L^+$ by assumption) and the other contains all ends of leaves in $U_t$, so $(A,B)$ contains ucountably many ends of leaves as announced. 

Now assume that $A$ and $B$ are in distinct connected components of $\RR^2\setminus U_t$.
Then there is a line $\Ga$ whose left end is $B$, whose right end is $A$ and whose intersection with 
$U_t$ is $\sigma([-t,t])$.  As $L^+$ is in the interval $(A,B)$ so that $L^+\subset \De^+_\Ga$,  one deduces that all the positive half leaves $L^+_r$, $r\in[-t,t]$ are contained in the upper half plane $\De^+_\Ga$ and therefore are between $A$ and $B$. So the interval $(A,B)$ (and also $(B,A)$) is uncoutable which is what we announced.

Conversely, if there are uncountably many ends in $(A,B)$ one of them is the end of a regular leaf as non-regular leaves are countably many.

This proves the equivalence of the two first items. The third items implies trivialy the second, so we now prove that the second implies the third. 

Let $A$ and $B$ be two ends of leaves so that $(A,B)$ is at most countable. 
We consider a line $\delta$ with the folowing properties: 
\begin{itemize} 
\item $A$ and $B$ are the right and left ends of $\delta$, respectively,
\item $\delta\setminus (B\cup A)$ is a segment $\sigma$, consisting in finitely many transverse segments $a_0,\dots a_k$ and finitely many leaf segments $b_1,\dots,b_k$, with $a_0(0)\in B$ and $a_k(1)\in A$.
\end{itemize}

Let $\De=\De^+(\delta)$ be the upper half plane bounded by $\delta$ and corresponding to the interval $(A,B)$. 

Notice that no entire leaf may be contained in $\De$ otherwhise there would be uncountably many ends between $A$ and $B$. 

We consider the half leaves  $L^+_{0,t}$ entering in $\Delta$ through $a_0(t)$.  As there are only countably many end between $A$ and $B$, there is a sequence of $t_n\to 0$ so that  $L^+_{0,t_n}$  goes out of $\Delta$ through a point $\sigma(s_n)$. 
Note that the half leaves $L^+_{0,t}$, $t\in[t_{n+1},t_n]$ need to go out of $\Delta$. 

Thus every  $L^+_{0,t}$, $t\leq t_0$ goes out of $\Delta$ at a point $\sigma(s(t))$, where $t\mapsto s(t)$ is a decreasing function. Let $s_0$ be the limit $$s_0=\lim_{t\to 0}s(t).$$ 

Notice that a half  leaf entering in $\Delta$ though $a_0$ cannot go out $\Delta$ through $a_0$ because a transverse segment cuts a leaf in at most a point. Thus we deduce that $s_0$ belongs to some $a_i, i>0$. 

We consider the compact segments $I_t\subset L^+_{0,t}$ joining $a_0(t)$ to $\sigma(s(t))$. 
We consider
$$\limsup_{t\to 0} I_t.$$
It is a closed subset of $\RR^2$ consisting on $B$ 
and of whole leaves contained in $\Delta$ and of a half leaf  $\tilde B_1$ ending at $\sigma(s_0)$. We already noticed that no entire leaves may be contained in $\Delta$.  Thus this limit  consists in 
$B\cup \tilde B_1$.  As a consequence, the ends $B$ and $\tilde B_1$ are successive ends,  $\tilde B_1\in (A,B)$ and thus $(\tilde B_1,A)$ is at most countable too. 

We consider $B_1\subset \tilde B_1$ the half leaf starting at the last intersection point of $\tilde B_1$ with $\sigma$.  Note that $B_1$ starts at a point of some segment $a_i$, with $i>0$. 

Thus, if $B_1\neq A$ one may iterate the argument, getting successive 
half leaves $B_i$ starting at points of some transverse segment 
$a_{j(i)}$, where $i\mapsto j(i)$ is stricly increasing.  As there are finitely many segments $a_i$ one gets that this inductive argument needs to stop. In other words,  there is $i$ with $B_i=A$, ending the proof: $[A,B]=A=B_i,A_{i-1},\dots, B_1,B$. 

This proves that the second item is equivalent to the third. 
\end{proof}
The proof of Lemma~\ref{l.finite} proved, as a by product, the following:

\begin{lemm}\label{l.hyperbolic} Assume that $A$ and $B$ are successive ends of leaves, that is: the interval $(A,B)$ is empty.  Then, there is an  embedding of $\psi\colon [-1,1]\times [0,1]\to\DD^2_\cF$ so that:
\begin{itemize}
 \item the segments $\psi([-1,1]\times \{t\})$, $0\leq t <1$, are leaf segment
 \item $A=\psi([-1,0)\times \{1\})$ and $B=\psi((0,1]\times \{1\})$
 \item the point $\psi(0,1)$ is the point is $\SS^1_\cF$ end of both end $A$ and $B$.  
\end{itemize}
\end{lemm}
\begin{defi}The embedding $\psi\colon [-1,1]\times [0,1]\to\DD^2_\cF$ is called a \emph{hyperbolic sector}.
\end{defi}

We say that two half leaves $A,B$ are \emph{asymptotic} 
if $[A,B]$ or $[B,A]$ does not contain any end of regular leaf. 
We already proved next lemma:  
\begin{lemm} To be asymptotic is an equivalence relation in the set of ends of leaves of $\cF$. 

Each equivalence class is either finite or countable and is, as an ordered  set, isomorphic to an interval of $(\ZZ,<)$.
 
There are at most countably many non-trivial classes. 
\end{lemm}

We also already proved: 

\begin{lemm}Let $\cF$ be a foliation (possibly with singular points of saddle type).  Then two half leaves tend to the same point $\theta\in \SS^1_\cF$ if and only if they are asymptotic, and every half leaf arriving to $\theta$ belongs to their asymptotic class. 
\end{lemm}

In particular, if a point of $\SS^1_\cF$ is the limit of a regular end of leaf, it is the limit of a unique end of leaf.

Notice that points at infinity which are limit of a unique end of leaf may be the limit of a non-separated end of leaf as shows next example:
\begin{exem}\label{e.Cantor} Let $K\subset \RR$ be a Cantor set and  consider
$$\cP_K=\RR^2\setminus (K\times[0,+\infty)).$$
Thus $\cP_K$ is homeomorphic to $\RR^2$.

Let $\cF_K$ be the restriction to $\cP_K$ of the horizontal foliation on $\RR^2$ (whose leaves are the $\RR\times \{y\}$). Thus all the leaves of the form $I\times \{0\}$ where $I$ is a connected component of $\RR\setminus K$ are pairwise non separated from below.

However, any two distinct ends of leaves of $\cF_K$ tend to distinct points in $\SS^1_{\cF_K}$.
\end{exem}

\begin{rema}  Assume that $\cF$ is oriented. 

If $A_0,\dots A_k$ are successive ends of leaves, and assuming $A_0$ is a right half leaf, then $A_1$ is a left half leaf and $A_0$ and $A_1$ are not separated from above. 

Then $A_2$ is a right half leaf and $A_1$ and $A_2$ are not separated from below, and so on.

Thus, each non trivial classes of the asymptotic relation consists in alternately right and left ends of non-separated leaves, alternately from above and from below.

\end{rema}

\subsection{Points at infinity which are not limit of leaves: center-like points.}

In this section, foliations are assumed to be non-singular. 

\begin{rema}\label{r.center}
Let $\cF$ be a foliation of $\RR^2$ and $o\in \SS^1_\cF$ be a point so that $o=\bigcap_n (a_n,b_n)$, $n\in \NN$ where $a_n,b_n$ are the limit points of the two ends of a same leaf $L_n$.

Then $a_n$ and $b_n$ tends to $o$ and $o$ is not a limit point of an end of  leaf of $\cF$.
\end{rema}
\begin{proof}Consider $\De_n$ being the compact disc of $\DD^2_\cF$ whose boudary (as a disc) is $L_n\cup [a_n,b_n]$.  Then the $\De_n$ are totally ordered by the inclusion and   $o\in \bigcap_n \De_n$. If a leaf $L$ had an end on $o$, it should be contained in every $\De_n$ and hence contained in $\bigcap_n \De_n$. Thus the two ends of $L$ are distinct points in $\bigcap_n [a_n,b_n]$  contradicting the hypothesis.
\end{proof}

We say that a point $o\in \SS^1_\cF$ satisfying the hypothese of Remark~\ref{r.center} is a \emph{center-like point}. 

Here is a very simple example with this situation: 
\begin{exem}
The trivial horizontal foliation $\cH$ admits two center-like points at infinity which are the limit points of the (vertical) $y$ axis (transverse to $\cH)$.
\end{exem}

It is indeed easy to check that:
\begin{rema}Given any foliation $\cF$ of $\RR^2$, $\SS^1_\cF$ carries at least $2$ center-like points.
To see that, just consider the (decreasing) intersection of the closure in $\DD^2_\cF$ of the half planes $\De^{\pm}_L$ for a maximal chain (given by Zorn lemma) for the inclusion. 
\end{rema}

But the situation may be much more complicated, as shows next example. 

\begin{exem}\label{e.Weierstrass}
Consider a simple closed curve $\gamma=\gamma^+\cup\gamma^-$ of $\RR^2$ where $\gamma^+$ and $\gamma^-$ are the graphs of continuous functions $\varphi\colon [-1,1]\to[0,1]$ and $-\varphi$, respectively, where
\begin{itemize}
 \item $\varphi(-1)=\varphi(1)=0$,
 \item $\varphi(t)>0$ for $t\in(-1,1)$,
 \item the local maxima and minima of $\varphi$ are dense in $[-1,1]$ (some kind of Weierstrass function). 
\end{itemize}
Let $\De$ be the open disc bounded by $\gamma$ and endowed with the constant horizontal foliation $\cF$. 

Then $\SS^1_\cF=\gamma$ and any local maximum point of $\gamma^+$ and any local minimum of $\gamma^-$ are center-like points of $\SS^1_\cF$
\end{exem}

The aim of this section is to show that the situation of Example~\ref{e.Weierstrass} is in fact very common. 

\begin{lemm}\label{l.center-like}  Le $\cF$ be a foliation on $\cR^2$. Assume that the union of leaves which are non separated at their right side  is dense in $\RR^2$, and in the same way, that the union of leaves which are non separated at their left  side  is dense in $\RR^2$. 

Then the set of center-like points on $\SS^1_\cF$ is a residual subset of $\SS^1_\cF$. 
\end{lemm}
\begin{proof}Fix a metric on $\SS^1_\cF$. 
Let $\cO_n\subset \SS^1_\cF$ be the set of points belonging to an interval $(a,b)$ of length less than $\frac 1n$  where $a,b$ are both ends of a same leaf of $\cF$. 

We will proof that $\cO_n$ is a dense open subset of $\SS^1_\cF$.  Then $\bigcap_n O_n$ will be the announced residual subset. 

The fact that $\cO_n$ is open is by definition. We just need to prove the density of $\cO_n$. 

Recall that the ends of regular leaves are dense in $\SS^1_\cF$. Thus we just need to prove that the ends of regular leaves are contained in the closure of $\cO_n$.

Let $L$ be a regular leaf and $\sigma\colon [-,1,1]\to \RR^2$ be a positively oriented transverse segment with $\sigma_0\in L$. We denote $L_t$ the leaf through $\sigma_t$ and ve recall that, as $L$ is regular, the rignt and left ends $L^+_t,L^-_t$ of $L_t$ tend to the right and left ends $L^+$ annd $L^-$, respectively, as $t\to 0$. 

Given any $r<s\in [-1,1]$, we denote by $U_{r,s}, U^{right}_{r,s}$, and $U^{left}_{r,s}$ the strip bounded by $L_r$ and $L_s$, and the two closed half strips obtained by cutting $U_{r,s}$ along the segment $\sigma([r,s])$. Let $I^{right}_{r,s}\subset \SS^1_\cF$ and $I^{left}_{r,s}\subset \SS^1_\cF$ be the corresponding intervals on $\SS^1_\cF$. Notice that, as $L$ is regular, these interval have a length smaller than $\frac 1n$ if $r,s$ close to $0$.

Our hypotheses imply that there are $t^{right},t^{left}\in (r,s)$ so that $L_{t^{right}}$ is non-separated at the right, and $L_{t^{left}}$ is non-separated at the right.

This implies that both $U^{right}_{r,s}$, and $U^{left}_{r,s}$ contain entire leaves. Thus  $I^{right}_{r,s}$ and $I^{left}_{r,s}$ contain intervals whose both extremal points are ends of the same leaf. Taking $r,s$ small enough, these intervals are contained in $\cO_n$ showing that the points of $\SS^1_\cF$ corresponding to $L^+$ and $L^-$ are in the closure of $\cO_n$. This ends the proof.
\end{proof}

However, not every point $o$ which is not limit of an end of leaf is center-like.
\begin{exem}\label{e.Cantor2}  Let $\cF_K$ be the foliation defined in Example~\ref{e.Cantor} by restriction of the horizontal foliation on $\RR^2\setminus (K\times [0,+\infty))$ where $K$ is a  Cantor set $K\subset \RR$. Consider a point $x\in K$ which are not the end point of a component of $\RR\setminus K$.  Then the point $(x,0)$ corresponds to a point in $\SS^1_\cF$ which is not limit of an end of leaf, and is not center-like.
\end{exem}

Consider a point $o\in \SS^1_\cF$ and assume it is not the limit point of any end of leaf. For any leaf $L$ we denote by $\De_L\subset \DD^2_\cF$ the compact disk containing $o$ and whose frontier in $\DD^2_\cF$ is the segment $\bar L$ closure of $L$. Then  $\De_L\cap \SS^1_L$ is a segment $I_L$ whose end points are the limit points of the ends of $L$. Note that the closed segment $I_L$ are totally ordered for the inclusion, and so does the disks $\De_L$. Let denote
$$I_o=\bigcap_L I_L \mbox{ and } \De_o=\bigcap_L \De_L$$
Then
\begin{itemize}
 \item if $o=I_o$ then $o$ is a center-like point.
 \item Otherwise, $\partial \De_o\cap\mathring{\DD^2_\cF}$ consists in infinitely (countably) many leaves pairwise not separated and there is a subsequence of them whose limit is $o$.
\end{itemize}

\section{The circle at infinity of a countable family of foliations}

The aim of this section is the proof of Theorem~\ref{t.countable-singular}, that is to build the compactification associated to a countable family of foliations with saddles and prove its uniqueness. 

\begin{rema}Example~\ref{e.uncountable} already shown us that Theorem~\ref{t.countable-singular} is wrong for uncountable families. 
\end{rema}

 The new difficulty in comparison to Theorem~\ref{t.feuilletage} is that there are no more separating families for the set of ends of all the foliations. 
 
 \begin{exem}Consider the restriction of the constant horizontal and vertical foliations to the strip $\{(x,y), |x-y|<1\}$, so important for the study of Anosov flows.  Then every end of horizontal leaf has a unique successor or predecessor which is the end of a vertical leaf. Thus no family can be separating. 
 \end{exem}

For by-passing this difficulty,  we will apply Theorem~\ref{t.union} instead of Theorem~\ref{t.rays}.

\begin{proof}[Proof of Theorem~\ref{t.countable-singular}] The ends of regular leaves $\cR=\bigcup_{i\in I}\cR_i$ of all the foliations $\cF_i, i\in I$  is a family of disjoints ends of rays. 

We have seen that for every foliation $\cF_i$  the set of ends of regular leaves $\cR_i$ admits a countable separating family, for instance by considering regular leaves through a dense subset in $\RR^2$. 

Thus Theorem~\ref{t.union} provides a  compactification of $\RR^2$ by $\DD^2$ satisfying the announced properties for the regular leaves, that is, items 2 and 3. 

For item 1 one need to see that even the ends of non regular leaves tends to points at infinity. That is given by Lemma~\ref{l.union}. 

The uniqueness comes from the uniqueness in Theorem~\ref{t.union}, ending the proof. 
\end{proof}

\subsection{Example: Countable families of polynomial vector fields} 
\begin{coro}\label{e.algebraic} Let $\cF=\{\cF_i\}_{i\in I}, I\subset \NN$ be a countable family of foliations directed by polynomial vector fields on $\RR^2$ whose singular points are all of saddle type.   Then the ends of leaves either are disjoint or coincide.  

Thus, according to Theorem~\ref{t.countable-singular}, there is a unique compactification $\DD^2_\cF=\RR^2\cup\SS^1_\cF$ for which the ends of regular leaves of the same foliation tend to pairwise distinct points at the circle at infinity, and this ends of leaves are dense in $\SS^1_\cF$. 
\end{coro}

\begin{proof} We just need to prove it for $2$ such distinct foliations $\cF$ and $\cG$. Consider the tangency locus of $\cF$ and $\cG$.  That is an algebraic set in $\RR^2$ which is either $\RR^2$ (so that $\cF=\cG$ contradicting the assumption)  or is  at most $1$-dimensional.   Thus it consist in the union of a compact part and a family of disjoint rays $\delta_1,\dots,\delta_k$.

If it is compact, then every end of leaf of $\cF$ is transverse to $\cG$ and therefore cuts every leaf of $\cG$ in at most $1$ point: the ends are disjoints. 

Otherwize, each ray $\delta_i$ either is tangent to both foliations and is therefore a comon leaf (which is one of the announced possibilities) or is transverse to $\cF$ and to $\cG$ out of a finite set (because the tangencies on $\delta_i$ are  a algebraic subset).  

Thus up to shrink the non-tangent $\delta_i$, we assume that they are transverse to both foliations therefore cut every leaf of $\cF$ in at most $1$ point.   This implies that every end of leaf of $\cF$ which is not  an end of $\cG$ is transverse to $\cG$ and thus is disjoint from any end of leaf of $\cG$, concluding. 
\end{proof}

\begin{rema}The compactification in Example~\ref{e.algebraic} is in general distinct from the algebraic extension of the $\cF_i$ on $\RR\PP^2$: for instance, consider the trivial example of $\RR^2$ endowed with the horizontal and vertical foliations. In this case  the compactification by the algebraic extension, all the leaves of the horizontal (resp. vertical) foliations tend to the same point at $\RR\PP^1$ (which corresponds to $2$ points for the circle at infinity).
\end{rema}

\subsection{Projections of $\DD^2_{\cF}$ on $\DD^2_{\cF_i}$ and center-like points on the circle at infinity}

 \begin{exem}Consider $\RR^2$ endowed with the trivial horizontal and vertical foliation, $\cH$ and $\cV$ respectively. Then the compactification $\DD^2_{\cH,\cV}$ is conjugated to the square $[-1,1]^2$ endowed with the trivial horizontal and vertical foliation.
 Every point $p\in \SS^1_{\cH,\cV}$, but the four vertices, are limit of exactly $1$ end of leaf, either horizontal (for $p$ in the vertical sides) or vertical (for $p$ in the horizontal sides).

 The projection $\Pi_\cH\colon\DD^2_{\cH,\cV}\to\DD^2\cH$ consists in colapsing the two horizontal sides, which are tranformed in the center-like points of $\SS^1_\cH$.

 The projection $\Pi_\cH\colon\DD^2_{\cH,\cV}\to\DD^2\cV$ consists in colapsing the two vertical sides, which are tranformed in the center-like points of $\SS^1_\cH$.
 \end{exem}

 \begin{exem}Consider the strip $\{(x,y)\in\RR^2, |x-y|<1\}$ endowed with the horizontal and vertical foliations $\cH$ and $\cV$ respectively.. Then $\DD^2_{\cH,\cV}=\DD^2_\cH=\DD^2_\cV$ and consists in adding to two points $\pm\infty$ to the closed strip $\{(x,y)\in\RR^2, |x-y|\leq 1\}$.  Every point in the sides $|x-y|=1$ are the limit of exactly $1$ end of leaf of $\cH$ and $1$ end of leaf of $\cV$, and the points $\pm\infty$ are center like for both foliations.
 \end{exem}

 These two examples show that pairs of very simple foliations may lead to different behavior of the projection of the compactification associated to the pair on the compactification of each foliation.

 Proposition~\ref{p.proj} below shows that, for complicated foliations, the compactification of the pair of foliations in general coincides with the compactification of each foliations.

\begin{prop}\label{p.proj} Let $\cF$, $\cG$ be two transverse foliations on $\RR^2$.  Assume that 
\begin{itemize}\item the union of leaves of $\cG$ which are not separated at their right from an other leaf is dense in $\RR^2$
 \item the union of leaves of $\cG$ which are not separated at their left from an other leaf is dense in $\RR^2$.
\end{itemize}
Then the identity map on $\RR^2$ extend as a homeomorphism from
$\DD^2_{\cF,\cG}\to\DD^2_\cF$: in other words 
$\DD^2_{\cF,\cG}=\DD^2_\cF.$
\end{prop}
\begin{proof} Assume that there is an open interval $I$ of $\SS^1_{\cF,\cG}$ corresponding to no end of leaf of $\cF$.  Then the ends of leaves of $\cG$ are dense in $I$, and therefore the projection of $I$ on  $\SS^1_\cG$ is injective.

Now Lemma~\ref{l.center-like} implies that there are leaves $L$ of $\cG$ having both ends on $I$. Thus up to change positive in negative, every positive half leaf of $\cF$ through $L$ has its end on $I$ contradicting the definition of $I$. 

So the points of $\SS^1_{\cF,\cG}$ corresponding to ends of leaves of $\cF$ are dense. Thus $\SS^1_{\cF,\cG}=\SS^1_\cF$, concluding the proof. 
\end{proof}

As a direct corollary of Proposition~\ref{p.proj} and Lemma~\ref{l.center-like} one gets
\begin{coro}Let $\cF,\cG$ be two transverse foliations on $\RR^2$ so that both $\cF$ and $\cG$ have density of leaves non separated at the right and of leaves non separated at the left. 

Then generic points in $\SS^1_{\cF,\cG}=\SS^1_\cF=\SS^1_\cG$ are center-like for both foliations. 
\end{coro}

\subsection{Hyperbolic sectors}
In the case of $1$ foliation we have seen that, if several ends of leaves have the same limit points on the circle at infinity, then they are ordered as a segment of $\ZZ$ and two succesive ends bound a hyperbolic sector. These hyperbolic sectors have a very precise model, which allows us to understand the position of a transverse foliation.

\begin{lemm}\label{l.tranverse-sector} Let $\cF$ and $\cG$ be two transverse foliations on $\RR^2$. and consider
$\pi_\cF\colon \DD^2_{\cF,\cG}\to \DD^2_\cF$.  Assume that $p\in \DD^2_\cF$ is the corner of a hyperbolic sector bounded by the ends $A$ and $B$ of leaves of $\cF$.

Then there is a non-empty interval $I_\cG$ of ends of leaves of $\cG$ ending at $p$ in $\DD^2_\cF$. Furthermore
\begin{itemize}
 \item either $I_\cG$ consist in a unique end of leaf $C$ of $\cG$ and $A,B,C$ tend to tend same point at infinity in $\DD^2_{\cF,\cG}$
 \item or $\pi_\cF^{-1}(I_G)$ is a closed interval on the circle $\SS^1_{\cF,\cG}$ whose interior consist in regular ends of leaves of $\cG$.
\end{itemize}

\end{lemm}
\begin{proof}Just use the model $[-1,1]\times [0,1]$ where $p$ is the point $(0,1)$, $A=[-1,0)\times \{1\}$ and $B=(0,1]\times \{1\}$, and the horizontal segment $[-1,1]\times \{t\}$, $0\leq t<1$ are $\cF$-leaf segments.  We can choose this model so that the vertical sides $\{-1\}\times[0,1]$ and $\{1\}\times[0,1]$ are leaves segments of $\cG$.  Consider the $\cG$-leaves throug $[-1,1]\times \{0\}$.  The leaves reaching $A$ and the leaves reaching $B$ are two non-empty intervals, open in $[-1,1]$ and disjoint.  By connectedness, there are leaves, corresponding to an closed interval of $[-1,1]$  which reach neither $A$ nor $B$ and these leaves end at $p\in \SS^1_\cF$. 

Assume that this interval is not reduced to a single end of leaf of $\cG$ and consider an end $C$ in the interior of this interval and assume that $C$ is, for instance, a right end. Consider the neighborhoods $U^{right}_t$ of $C$ defined in Section~\ref{s.background}. Then $\bigcap_t U^{right}_t$ consist in $C$ and in a (maybe empty) set of entire leaves of $\cG$ contained in the hyperbolic sector. Assume that this set is not empty and let $D$ be such a leaf of $\cG$.  Every leaf of $\cF$ cutting $D$ has an half leaf contained in the hyperbolic sector, contradicting the definition of hyperbolic sector. Thus $C=\bigcap_t U^{right}_t$, meaning that $C$ is a regular end of leaf of $\cG$, ending the proof.

\end{proof}

As a straightforward consequence, one gets:

\begin{coro}\label{c.hyperbolics} Let $\cF=\{\cF_i\}_{i\in I}$, $I\subset \NN$  be an at most countable family of pairwise  transverse foliations on $\RR^2$.  Consider a point  $p\in \DD^2_{\cF}$. Then
\begin{itemize}
 \item either at most $1$ end of leaf of each $\cF_i$ has $p$ as its limit.
 \item or the set of ends tending to $p$ is ordered as an interval of $\ZZ$ and, between any two succesive ends of leaves of the same $\cF_i$, there is exactly $1$ end of a leaf of each $\cF_j$, $j\neq i$.
\end{itemize}
\end{coro}




\section{ The circle at infinity for  orientable laminations.\label{s.laminations}}

\subsection{The circle at infinity of a lamination}
The way we proposed to compactify $\RR^2$ can be generalized for any object providing a family of rays admitting a separating set. 

For instance, what about laminations?  The theory cannot be extended without hypotheses. An evident obstruction is that the leaves can be too few for going to a dense subset of a circle at infinity. But there are more subtle issues as shows Example~\ref{e.plykin} below.

\begin{exem}\label{e.plykin} There are closed laminations 
whose leaves are recurrent.  For instance consider a Plykin attractor on $\RR^2$: it is a compact minimal lamination (by the unstable manifolds). 

If you consider now  a Plykyn attractor on $\SS^2=\RR^2\cup\{\infty\}$ where $\infty$ belongs to the attractor, we get a closed lamination on $\RR^2$ where every leaf is unbounded but recurrent. 
\end{exem}
Notice that the recurrent lamination in Example~\ref{e.plykin} are not orientable. Let me show that Poincar\'e Bendixson argument applies on orientable laminations: 

\begin{lemm}\label{l.lamination-PB} Let $\cL$ be a closed orientable lamination of $\RR^2$. 
Given any  leaf $L$,  either the closure  $\bar L$ contains a compact leaf or $L$ is a line.
\end{lemm}
\begin{proof}  
If $\bar L=L$ then $L$ is either a compact leaf or is a line (i.e. is properly embedded in $\RR^2$). 
Assume now that $\bar L\setminus L$ contains a point $x\in\RR^2$. We fix an orientation of $\cL$. Chose a segment $\sigma\colon [-1,1]$ transverse to $\cL$ so that $\sigma(0)=x$ and so that $\sigma$ cuts positively every leaf. The hypothesis implies that $L$ cuts $\sigma$ in an infinite set. Consider $2$ successive (for the order in the leaf $L$) intersection points $z_0,z_1$. Then one gets a simple closed curve $\delta$ in $\RR^2$ by concatenation of the segments $[z_0,z_1]_L$ and $[z_1,z_0]_\sigma$  joining $z_0$ to $z_1$,  in $L$ and in $\sigma$ respectively. 

Consider the disc $\De$ bounded by $\delta$.  Every leaf cuts $\delta$ with the same orientation, that is, either every leaf enter in $\De$ or every leaf goes out of $\De$. Up to reverse the orientation one may assume that every leaf enters in $\De$ and in particular $L$ enters in $\De$. In other words, there is a positive half leaf $L_+$ included in $\De$.   This half leaf cannot be reccurent (otherwize it would cut again $[z_0,z_1]_\sigma$ and for that it needs to go out of $\De$). Futhermore:

\begin{clai}\label{cl.PB} No other leaf $L'\neq L$ can accumulate on $L$:  $L\cap \bar L'=\emptyset$ if $L'\neq L$. 
 
\end{clai}
\begin{proof}If $L'$ accumulates on $ L$, it cuts $[z_0,z_1]_\sigma$ on an infinite set. 
\end{proof}

Thus  the $\omega$-limit set is $\omega(L)=\bar L_+\setminus L_+$ and is not empty.  Consider $y\in \bar L_+\setminus L_+$. The leaf $L_y$ is contained in $\De$.  Either $L_y$ is compact and $\omega(L)= L_y$ and we are done, or $\bar L_y\setminus L_y\not \emptyset$.  In that case, Claim~\ref{cl.PB} implies that $L_y$ is not accumulated by any leaf, in particular by $L$, contradicting the definition of $L_y$. 
\end{proof}

We are now ready to extend Theorem~\ref{t.feuilletage}  to the case of orientable laminations:

\begin{theo}\label{t.lamination} Let $\cL$ be a closed orientable lamination of $\RR^2$ with no compact leaf and assume that the set of leaves of $\cL$ is uncountable.  Then there is a compactification $\DD^2_\cL\simeq $ of $\RR^2$  by  adding a circle at infinity $\SS^1_\cL=\partial\DD^2_\cF$ with the following properties: 
 \begin{itemize}
 \item any half leaf tends  to a point at infinity. 
  \item given a point $\theta\in \SS^1_\cL$ the set of ends of leaves tending to $\theta$ is at most countable. 
  \item for any non-empty open subset $I$ of $\SS^1_\cL$ the set of points in $I$  corresponding to limits of ends of leaves is uncountable.
 \end{itemize}
 
 Furthermore this compactification of $\RR^2$ by $\DD^2$ with these three properties is unique, up to a homeomorphism of the disk $\DD^2$. 
 
\end{theo}

Let me just given a sketch of proof. 
\begin{proof} The lamination $\cL$ is assumed to be oriented and without compact leaves, so that every leaf is a line, according to Lemma~\ref{l.lamination-PB}.  

 According to Cantor-Bendixson theorem, see for instance \cite{Ke}, the lamination $\cL$ can be written in a unique way as union $\cL=\cL_0\cup\cL_1$ of two disjoint laminations, where $\cL_0$ is a 
 closed lamination with no isolated leaves and $\cL_1$ consists in a countable set of leaves. 
 
 A leaf $L\in \cL_0$ is called \emph{regular} if it is accumulated on both sides by leaves in $\cL_0$ and is separated from any other leaf of $\cL_0$.  The same proof as for foliations shows that the set of leaves in $\cL_0$ which are not regular are at most countable.
 
Finally, as for foliations, one consider the set $\cR$ of germs of rays contained in regular leaves of  $\cL_0$.  We consider a countable set $\cD$ of regular leaves, whose union is dense in $\cL_0$, and as for the case of foliations, one proves that the  rays in $\cD$ are separating for $\cR$. 

Then we apply Theorem~\ref{t.rays} and we get the announced canonical compactification. 
\end{proof}

When $\cL$ is transversely a perfect compact set (that is, there is a transverse segment $\sigma$ through every point $x\in\cL$ so that $\sigma\cap \cL$ is a compact set without isolated points), then the compactification given by Theorem~\ref{t.lamination} seems very natural: any homeomorphism $h$ of $\RR^2$ preserving $\cL$ extends on $\SS^1_\cL$ as a homeomorphism $H$ of $\DD^2_\cL$, and the restriction $H|_{\SS^1_\cL}$ is the identity map if and only if $h$ preserves every leaf of $\cL$. That is no more the case if $\cL$ has isolated leaves. 

For lamination with isolated leaves, Theorem~\ref{t.lamination} just ignores the countable part $\cL_1$ of $\cL$ (in the Cantor Bendixson decomposition of $\cL$). We will now  propose a cannonical compactification which takes in account this countable part. 

We start by looking at  two very different examples of countable oriented laminations. 

\begin{exem}Consider the lamination $\cL= \RR\times \ZZ$.  Then $\cL$  does not admit any compactification by a circle at infinity so that any homeomorphism $h$ preserving $\cL$ extends on the circle at infinity. 
\end{exem}

\begin{exem} \label{e.denombrable} Consider a hyperbolic surface  $S$ of finite volume and consider a set $\ell$  of essential disjoint simple closed geodesic on $S$. Then the lift $\cL$ of $\ell$ on the universal cover $\tilde S=\mathring \DD^2$ is a countable, discrete lamination by geodesic of the Poincaré disc so that the ends of leaves tend each to  points on the circle $\SS^1=\partial \DD^2$, and the set of such limit points is dense on $\SS^1$ as the action of $\pi_1S$ on $\SS^1$ is minimal. 

In this example, the lamination is  transversely discrete, but the set of ends of leaves is a separating set for himself for the cyclic order. 
\end{exem}

In the example~\ref{e.denombrable}  above, what implies the existence of a separating set is the minimality of the action on the circle at infinity of a natural compactification. 

In order to propose a cannonical compactification for a closed, oriented lamination without compact leaves we need to determine what part of a cyclically totally ordered set admits separating subset.  That is what is done in next easy proposition whose proof is let to the reader. 

\begin{prop}\label{p.denombrable-separant} Let $X$ be a set endowed with a total cyclic order.  Consider the relation on $X$ defined as follows: $x\approx y$ if one of the intervals $[x,y]$ or $[y,x]$ does not contained any  self-separating subset $E$ (with $\#E\geq 2$).  Then 
\begin{itemize}
 \item the relation $\approx$ is an equivalence relation
 \item each equivalence class is an interval
 \item the cyclic order on $X$ induces a total cyclic order on the quotient $X/\approx$ 
\end{itemize}

Futhermore, $X/\approx$ is either a single point or is an infinite self-separating set. 
\end{prop}

Note that any two distinct points in a self separated set belong to distinct classes, so that $\# (X/\approx)=1$ if and only if $X$ does not contain any (non-trivial) self-separating subsets.  Otherwise $\# (X/\approx)=\infty$. 

The canonical compactification is now given by Theorem~\ref{t.Lami} below: 
\begin{theo}\label{t.Lami} Let $\cL$ be a closed oriented lamination of the plane $\RR^2$, with no compact leaf, and let $\cR$ be the set of ends of leaves of $\cL$. As the ends of leaves are disjoints rays the set $\cR$ is totally cyclically ordered. Assume that $\#(\cR/\approx)>1$. 

Then there is a unique compactification $\DD^2_\cL$ of $\RR^2$ by adding a circle at infinity $\SS^1_\cL$ so that 
\begin{itemize}
 \item any end of leaf of $\cL$ tends to a point in $\SS^1_\cL$
 \item the set of points in $\SS^1_\cL$ end of an end of leaf is dense in $\SS^1$
 \item two ends of leaves tend to the same point in $\SS^1_\cL$ if and only if they belong to the same class in $\cR/\approx$. 
\end{itemize}
\end{theo}
\begin{proof} Just apply Theorem~\ref{t.rays} to a subset $E\subset \cR$ containing exactly $1$ representative in each class of $\approx$. One checks that the compactification obtained satisfies the announced properties and does not depend on the choice of $E$. 
\end{proof}
\begin{rema} Every class of $\approx$ in $\cR$ is at most countable because the set of ends of regular leaves in the perfect part $\cL_0$ is self-separating.  
\end{rema}

This compactification takes in account more leaves that the compactification given by Theorem~\ref{t.lamination}, but it is still may have unexpected behaviours:    

\begin{exem}\label{e.denombrables} Consider a non-compact hyperbolic surface  $S$ of finite volume and consider a closed lamination $\ell$  defined by  two disjoint freely homotopic essential closed curves and a closed (but non-compact) leaf whose both ends tend to the same puncture of $S$. 

Then the lift $\cL$ of $\ell$ on the universal cover $\tilde S=\mathring \DD^2$ is a countable, discrete lamination of the Poincaré disc so that the ends of leaves tend each to  points on the circle $\SS^1=\partial \DD^2$,  the set of such limit points is dense on $\SS^1$ (again as the action of $\pi_1S$ on $\SS^1$ is minimal). 

In this example, however, there are pairs of leaves which share the same limits of their both ends, and there are leaves whose both ends tend to the same point.  
\end{exem}

Given a closed oriented lamination $\cL$ with no compact leaves and its Cantor-Bendixson decomposition $\cL=\cL_0\cup\cL_1$ ($\cL_0$ is a closed lamination without isolated leaves and $\cL_1$ is countable),   
Theorem~\ref{t.Lami} takes in account the part of the ends of leaves in $\cL_1$ with separating subsets, in contrast with Theorem~\ref{t.lamination}. 
For my personal taste, the main issue in Theorem~\ref{t.Lami} is that I did not found any natural criterion to calculate the equivalence classes of $\approx$ in $\cL_1$.  In fact Lemma~\ref{l.denombrable-separe} below seems to present as paradoxal the fact that $\cL_1$ may have separating subsets: 
\begin{lemm}\label{l.denombrable-separe} Let $D\subset \RR$ be a countable compact subset, ordered by $\RR$. Then $D$ does not contain any self separating subsets $\cE\subset D$(that is, $\# \cE>2$ and  for every $x<z$, $x,z\in \cE$ there is $y\in \cE$  with $x<y<z$) .  
\end{lemm}
\begin{proof}If $\cE$ is a non-trivcial sel-spearting subset, then there is an increasing bijection from $\cE\setminus \{\min \cE,\max\cE\}$ to $\QQ\cap (0,1)$. This increasing bijection extends in a unique way in a (non-strictly) increasing map from $\RR \to [0,1]$. This map is continuous and the image of $D$ is $[0,1]$. 

Thus $D$ is not countable.  
\end{proof}

Lemma~\ref{l.denombrable-separe} tell us that the separating property of a closed countable 
lamination cannot be obtained locally (in foliated charts of the lamination).  One deduces:

\begin{prop} Let $\cL$ be a closed countable orientd lamination of $\mathring{\DD^2}$ so that every end of leaf tends to a point on $\SS^1$ and the set of such limit points  are dense in $\SS^1$.   

Then given any non-empty open interval $I\subset \SS^1$ there is $L\in\cL$ whose both ends have their limits in $I$. 

More precisely, any neighborhood of $I$ in $\DD^2$ contains an entire leaf of $\cL$. 
\end{prop}
\begin{proof} One consider a neighborhood of $I$ bounded by two half leaves whose limits are points $x\neq y \in I$ and a segment $\sigma$ transverse to $\cL$ and joining this two leaves. If no leaves is contained in this neighborhood, then every leaf having an end in $I$ cuts $\sigma$. 

On the other hand any dense subset of an interval $J$ of $\RR$ is self separating. One deduces that $\sigma\cap \cL$ contains a self separating subset, but is is a countable compact set, and this contradicts Lemma~\ref{l.denombrable-separe}. 
 
\end{proof}

This proposition says that the separating property for a countable oriented lamination is obtained by leaves in small neighborhoods of the points at infinity.

\subsection{Families of transverse laminations}

Transversality does not imply in general the compactness of the intersection of two leaves of transverse laminations. But this compactness is our main hypothesis for  the compactification associated to families of foliations. 

However, if two lines $L_1,L_2\subset\RR$ intersect always with the same orientation, then $\# L_1\cap L_2\leq 1$.  One deduces that Theorem~\ref{t.foliations} extends without difficulties to countable families of oriented closed laminations intersecting pairwise transversely and with always the same orientation

\begin{theo}\label{t.laminations} Let $\cL=\{\cL_i\}$, $i\in I\subset \NN$ be a family of closed orientable laminations of $\RR^2$ with no compact leaves and so that the set of leaves of $\cL_i$ is uncountable.  We assume that the laminations are pairwize transverse with constant orientation of the intersections. 
Then there is a compactification $\DD^2_\cL\simeq $ of $\RR^2$  by  adding a circle at infinity $\SS^1_\cL=\partial\DD^2_\cF$ with the following properties: 
 \begin{itemize}
 \item any half leaf tends  to a point at infinity. 
  \item given a point $\theta\in \SS^1_\cL$ the set of ends of leaves tending to $\theta$ is at most countable. 
  \item for any non-empty open subset $I$ of $\SS^1_\cL$ the set of points in $I$  corresponding to limits of ends of leaves is uncountable.
 \end{itemize}
 
 Furthermore this compactification of $\RR^2$ by $\DD^2$ with these three properties is unique, up to a homeomorphism of the disk $\DD^2$. 
 
\end{theo}

\section{Actions on a bifoliated plane}\label{s.action}

We have seen than any homeomorphism $h\in Homeo(\RR^2)$ preserving a at most countable family of transverse foliations $\cF$ admits a unique extension  as an homeomorphism on the compactification $\DD^2_\cF$.  

Thus if $H\hookrightarrow Homeo(\RR^2)$ is a group acting on $\RR^2$ and preserving the (at most countable) family of transverse foliations $\cF$ then this action extends in an action on $\DD^2_\cF$.  By restriction to the circle at infinity, one gets an action of $H$ on $\SS^1_\cF$.

If $H\hookrightarrow Homeo(\RR^2)$ is a group acting on $\RR^2$ and preserving a  family of foliations $\cF$, we say that \emph{the action is minimal on the leaves of $\cF$} if $H(L)$ is dense of $\RR^2$ for every leaf $L$ of a foliation of the family $\cF$.  

\subsection{Faithfullness}

\begin{prop}\label{p.quasi-faithfull} Let $\cF$ be a foliation, and $h\in Homeo(\RR^2)$ be a homeomorphism preserving $\cF$.  Then 
the action of $h$ on $\SS^1_\cF$ is the identity map if and only if $h(L)=L$ for any leaf $L$, and $h$  preserves the orientation of the leaves. 
\end{prop}
\begin{proof} If $h$ preserves every leaf and its orientation, then it preserves any limit of its ends.  As these limit of ends are dense in $\SS^1_\cF$ one gets
that the homeomorphism induced by $h$ on $\SS^1_\cF$ is the identity map. 

Conversely, if $h$ induces the identity on $SS^1_\cF$ then for every leaf $L$ the leaf  $h(L)$ have the same limit of ends as $L$.  According to Lemma~\ref{l.injective} this implies $h(L)=L$ as announced. 
\end{proof}

\begin{coro}\label{c.faithfullness} Let $\cF=\{\cF_i\}, i\in I\subset \NN$ be a family of at least $2$ transverse foliations. Let $h\in Homeo(\RR^2)$ be a homeomorphism preserving each foliation $\cF_i$.  Then 
the action of $h$ on $\SS^1_\cF$ is the identity map if and only if $h$ itself is the identity map. 
\end{coro}
\begin{proof} If the induced homeomorphism induced by   $h$ on $\SS^1_\cF$ is the identity map then the same happens to homeomorphism induced by $h$ on every $\SS^1_{\cF_i}$ (because they are quotient of $\SS^1_\cF$).  Thus Proposition~\ref{p.quasi-faithfull} implies that $h$ preserves each leaf of each $\cF_i$.  As every point of $\RR^2$ is the unique intersection point of the leaves through it, one deduces that every point of $\RR^2$ is fixed by $h$ and $h$ is the identity map. 
 
\end{proof}

\subsection{Orientations and injectivity of the projections}
Let $\cF$ be a foliation of the plane $\RR^2$, endowed with an orientation and a transverse orientation. Let $\cG\subset Homeo(\RR^2)$ be a group of homeomorphisms preserving (globally) $\cF$. Let $\cG^+$ (resp.  $\cG_+$) be the index at most $2$  subgroup consisting of the elements of $\cG$ preserving the orientation (reps. the transverse orientation)of $\cF$, and $\cG^+_+=\cG^+\cap\cG_+)$ the subgroup of elements preserving both orientations.   Then:

\begin{lemm}\label{l.orientation} If one of the groups $\cG,\cG^+,\cG_+,\cG^+_+$ acts minimally on the leaves of $\cF$, then so does each of these $4$ groups.
\end{lemm}

We will indeed prove Lemma~\ref{l.finite-index}  for which Lemma~\ref{l.orientation} is a particular case.

\begin{lemm}\label{l.finite-index} Let $\cG$ be a group acting minimally on the leaves of a foliation $\cF$ of $\RR^2$, and $\cH\subset \cG$ be a subgroup of finite index. Then $\cH$ acts minimally on the leaves of $\cF$.
\end{lemm}
\begin{proof}
 $\cG$ acts minimally on the leaves of $\cF$, and consider such a leaf $L$.
 As $\cH$ is a finite index subgroup, there are $g_1,.., g_n\in \cG$ so that for any $g\in \cG$ there is $i\in\{1,\dots, n\}$ with  $g.\cH=g_i\cH$. Let denote $\cH_i=g_i\cdot \cH$. In particular $\cG=\bigcup_i \cH_i$, and then  $\RR^2=\bigcup_i \overline{\cH_i(L)}$.

Consider any open subset $O$ of $\RR^2$.

$$O= O\cap \bigcup_i \overline{\cH_i(L)}= \bigcup_i (O\cap\overline{\cH_i(L)})$$
The open set $O$ is a baire space so that the union of finitely many  closed sets with empty interior has empty interior: one deduce that at least one of the $O\cap\overline{\cH_i(L)}$ have non empty interior.  One deduces that the union 
$
\bigcup_i\mathring{\overline{\cH_i(L)}}$ of the interiors of the $\overline{\cH_i(L)}$ is dense in $\RR^2$.

Notice that for every $i$ and every $g$ there is $j$ so that $g(\cH_i(L))=\cH_j(L)$.

Consider $\RR^2\setminus \bigcup_i\mathring{\overline{\cH_i(L)}}$.  It is a $\cG$-invariant closed set, saturated for the foliation $\cF$, and with empty interior.  As every $\cG$-orbit is dense, one deduces that this set is empty.

Thus $$\RR^2=
\bigcup_i\mathring{\overline{\cH_i(L)}}.$$

The open sets $\mathring{\overline{\cH_i(L)}}$ are images on from the other by homeomorphisms in $\cG$, and in particular they are all non-empty.

As $\RR^2$ is connected, one deduces that the open sets $\mathring{\overline{\cH_i(L)}}$ are not pairwise disjoint. Let $k\in \{1,\dots,n\}$ be the maximum number so that there are distinct  $i_1,\dots, i_k$ with

$$ \bigcap_1^k \mathring{\overline{\cH_{i_j}(L)}}\neq\emptyset.$$

As the $\mathring{\overline{\cH_i(L)}}$ are not pairwise disjoint, we know that $k\geq 2$.  We will prove, arguing by contradiction:
\begin{clai} $k=n$.  \end{clai}
\begin{proof} For that we assume that $k<n$.

Then we consider the union of all the intersections of $k$ of these open sets. This union is a $\cF$-saturated $\cG$-invariant non-empty set and hence is dense. Its complement is an $\cF$-saturated invariant closed set with empty interior,  and therefore is empty.

Thus $\RR^2$ is the union of these open sets. Now again the connexity of $\RR^2$ implies that these open sets are not pairwise disjoint. This provides a non-empty intersection of $2$ distinct of these sets, that is, a non empty intersection of more than $k$ of the $\mathring{\overline{\cH_i(L)}}$, contradicting the choice of $k$.  This shows $k=n$ proving the claim.
\end{proof}

Thus $$ \bigcap_1^n \mathring{\overline{\cH_{i}(L)}}.$$  is an non-empty, $\cG$-invariant open set saturated for the foliation $\cF$, and thus it is dense in $\RR^2$.

We just proved that $\cH(L)$ is dense in $\RR^2$, concluding the proof.
\end{proof}

We will use the next straightforward corollary of Lemma~\ref{l.orientation}
\begin{coro}\label{c.dense} Let $H\subset Homeo(\RR^2)$ be a group preserving a foliation $\cF$ and acting minimally on the leaves. Assume that $L$ is a leaf which is not separated at the right and from below. Then the union of the leaves $h(L)$, $h\in H$, which are non-separated at the right and from below is dense in $\RR^2$ (the same holds changing right by left and/or below by above).
\end{coro}

 As a direct consequence of Proposition~\ref{p.proj} and Corollary~\ref{c.dense} we get:
\begin{prop}\label{p.Hproj}
Let $\cF$,$\cG$ be two transverse foliations of $\RR^2$ and $H\subset Homeo(\RR^2)$ preserving $\cF$ and $\cG$. Assume that the orbit of every leaf of $\cG$ in dense in $\RR^2$.

If $\cG$ has a non-separated leaf, then
 the projection of $\Pi_\cF\colon \DD^2_{\cF,\cG}\to \DD^2_\cF$ is injective.
\end{prop}
\begin{proof}If $\cG$ has a non separated leaf $L_1$ at the right, it is non separated from a leaf $L_2$ which is non separated at the left. Now Corollary~\ref{c.dense} asserts that the leaves of $\cG$ non separated at the left as well as the leaves non separated at the right are dense in $\RR^2$. Now Proposition~\ref{p.proj} asserts that $\Pi_\cF$ is a homeomorphism, concluding.
\end{proof}

\subsection{Minimality of the action on the circle at infinity}

\begin{theo}\label{t.minimal}
Let $\cF$  be a foliation on the plane $\RR^2$ and $H\subset Homeo(\RR^2)$  preserving the foliation $\cF$.
\begin{enumerate}
\item If the action of $H$ on $\SS^1_{\cF}$ is minimal
then the foliation $\cF$  admits non separated leaves from above and non separated leaves from below.

\item Conversely if the foliation $\cF$  admits non separated leaves from above and non separated leaves from below
and if the orbit of every leaf is dense in $\RR^2$ then 
the action of $H$ on $\SS^1_{\cF}$ is minimal.
\end{enumerate}
\end{theo}
We will see with Theorem~\ref{t.non-transitive} that the minimality of the action on the leaves is not a necessary condition for the minimality of the action on the circle at infinity. 

Item 1 of Theorem~\ref{t.minimal} is a consequence of Proposition~\ref{p.onlyup} below.

\begin{prop}\label{p.onlyup} Let $\cF$ be a foliation of $\RR^2$ and assume that $\cF$ has no non-separated leaves from below.
Given any leaf $L$ we denote by $\De^+_L$ the closure on $\DD^2_\cF$ of the upper half plane of $\RR^2$ bounded by $L$. 

Then $\bigcap_{L\in\cF} \De^+_L$ is non empty and consists in an unique point $O_\cF$ on $\SS^1_\cF$. 
As a consequence, any $h\in Homeo(\DD^2_\cF)$ preserving  $\cF$ admits $O_\cF$ as a fixed point: 
$$h(O_\cF)=O_\cF.$$
\end{prop}
\begin{proof}
We introduce a  relation on the set $\cL$ of leaves of $\cF$ as follows: $L_1\preceq L_2$ if there is a positively oriented (for a transverse orientation of $\cF$) transverse segment $\sigma$ starting at $L_1$ and ending at $L_2$. One easily checks that $\preceq$ is a partial order relation on $\cL$. 

Due to the connexity of $\RR^2$, one gets:
\begin{clai}\label{cl.distance} given any leaves $L,\tilde L\in\cL$ there is $k\geq0$ and $L_0,\dots,L_k\in \cL$ so that,
\begin{itemize}\item for any $i\in\{0,\dots,k-1\}$ the leaves $L_i$ and $L_{i+1}$ are comparable for $\preceq$ (that is $L_i\preceq L_{i+1}$ or $L_{i+1}\preceq L_i$)
 \item $L=L_0$ and $L'=L_k$
\end{itemize}
\end{clai}
\begin{proof}There is a countable family of segments in $\RR^2$ transverse to $\cF$ and so that every leaf $L$ cuts at least one of these segments. The set of leaves cutting a given segment induces a connected open set of $\RR^2$.  Given any two points in $\RR^2$ one considers a compact path joining these two points. By compacity, it is covered by a finite family of these open sets. One concludes easily. 
\end{proof}
We denote $\prec L,L'\succ\in \NN$ the minimum value of such a number $k$. One easily checks that $\prec \cdot,\cdot\succ$ is a distance on the set of leaves $\cL$. 

Up to now, this could be done for any foliation $\cF$. In this setting, our hypothesis that $\cF$ does not admit leaves which are non-spearate from below is translated as follows: 
\begin{clai}\label{cl.up} Assume that $L_0,L_1,L_2\in \cL$ are three leaves so that $L_0\preceq L_1$ and $L_0\preceq L_2$.  Then $L_1$ and $L_2$ are comparable for $\preceq$. 
\end{clai}
\begin{proof} We assume that the leaves $L_i$ are distinct, otherwise there is nothing to do. 
Let $\sigma_i\colon [0,1]\to \RR^2$, $i=1,2$ transverse to $\cF$ and positively oriented so that $\sigma_i(0)\in L_0$ and $\sigma_i(1)\in L_i$. 

Let $I=\{t\in[0,1], L(\sigma_1(t))\cap \sigma_2\neq \emptyset\}$ and $J=\{t\in[0,1], L(\sigma_1(t))\cap \sigma_2\neq \emptyset\}$. As $\RR^2$ is simply connected, one shows that $I$ and $J$ are connected and each of them contains $0$. 

Let $t_1=\sup I$ and $t_2=\sup J$
For any $t\in[0,t_1)$ let $\tilde t\in J$ so that $L(\sigma_1(t))=L(\sigma_2(\tilde t)$. In particular, $\tilde t$ tends to $t_2$ as $t$ tends to $t_1$.

Thus the leaves $L(\sigma_1(t_1)$ and $L(\sigma_2(t_2)$ are accumulated from below by the leaves $L(\sigma_1(t))=L(\sigma_2(\tilde t)$, thus are non separated from below. By assumption on $\cF$ this implies that they are equal:
$$L(\sigma_1(t_1)=L(\sigma_2(t_2)$$

If $t_1<1$ and $t_2<1$ then the leaf $L(\sigma_1(t)$ for $t>t_1$ close to $t$ cuts $\sigma_1$ at a point $\sigma_2(\tilde t)$ with $\tilde t>t_2$, close to $t_2$.  This contradicts our choice of $t_1$ and $t_2$. 

Thus $t_1=1$ or (non exclusive) $t_2=1$.  In the first case $L_1=L(\sigma_1(t_1))$ cuts $\sigma_2$ and then $L_1\preceq L_2$ and in the second case $L_2$ cuts $\sigma_1$ and $L_2\preceq L_1$.  This ends the proof. 

\end{proof} 

 As a consequence of Claims~\ref{cl.distance} and~\ref{cl.up} one deduces
 \begin{claim}\label{cl.prec} Given any two leaves $L,\tilde L$ there is a leaf $\hat L$ so that $L\preceq \hat L$ and $\tilde L\preceq \hat L$.  In particular, the distance $\prec \cdot,\cdot\succ$ is bounded by $2$. 
 \end{claim}
 \begin{proof}Consider a finite sequence of leaves $L=L_0,\dots,L_k=\tilde L$, $k=\prec L,\tilde L\succ$, and 
 $L_i$ comparable with $L_{i+1}$. 
 
 The minimality of $k$ implies that $L_{i-1}$ and $L_{i+1}$ are not comparable (otherwise one could delete $L_i$ geting a strictly smaller sequence). 
 
 Assume that there is $i\in \{1,\dots k-1\}$ so that 
$ L_{i-1}\succeq L_i$.  If  $L_i\succeq L_{i+1}$ then $L_{i-1}\succeq L_{i+1}$ which is forbidden by the observation above.  Thus $L_i\preceq L_{i+1}$ and Clain~\ref{cl.up} implies again that  $L_{i-1}$ and $L_{i+1}$ are comparable, which again is impossible. 
This proves that 

$$\forall i\in \{1,\dots k-1\}, L_{i-1}\preceq L_i$$
 \end{proof}
 
 As a consequence one deduces
 \begin{clai}\label{cl.exhaustive} There is a increasing sequence $L_i\prec L_{i+1}$ $i\in\NN$, $L_i\in \cL$  so that, given any leaf $L\in\cL$ there is $n$ with $L\prec L_n$. 
 \end{clai}
 \begin{proof} One chose a countable set of compact positively  oriented segments $\sigma_i\colon[0,1]\to \RR^2$ transverse to $\cF$ to that any leaf cuts one of the $\sigma_i$ (and thus is less than $L(\sigma_i(1)$ for $\preceq$). Then one builds inductively the sequence $L_i$: $L_{i+1}$ is obtained by applying Claim~\ref{cl.prec} to the leaves $L_i$ and $L(\sigma_i(1)$. 
 \end{proof}

 \begin{clai}\label{cl.decreasing}  The compact discs $\De^+_L$ are decreasing with $L$ for $\prec$: more precisely, if $L\prec \tilde L$ then 
 $$\De^+_{\tilde L} \subset \mathring{\De^+_L},$$ 
 where $\mathring{\De^+_L}$ denotes the interior for the topology of $\DD^2_\cF$ ( it does not means the open disc). 
 \end{clai}
 \begin{proof}The hypothesis implies that $L_{i+1}$ is contained in the interior of $\De^+_{L_i}$, so that $\De^+_{L_{i+1}}\cap \RR^2$ is contained in the interior of $\De^+_{L_i}$.  We need to prove that $\De^+_{L_{i+1}}\cap \SS^1_\cF$ is contained in the interior of $\De^+_{L_{i}}\cap \SS^1_\cF$.  In other words, we need to prove that the ends of $L_{i+1}$ do not share a limit with the ends of $L_i$.  
 
 Recall that $L_i\prec L_{i+1}$ that is, there is a segment $\sigma\colon[0,1]\to \RR^2$ transverse to $\cF$ and positively oriented, with $\sigma(0)\in L_i$ and $\sigma(1)\in L_{i+1}$.  If $L_i$   share with $L_{i+1}$ a limit point in $\SS^1_\cF$ so does any leaf $L(\sigma(t))$ contradicting the fact that points in $\SS^1_\cF$ are limits of at most a countable set of ends of leaves. This ends the proof. 
 
 \end{proof}

 Thus Claims~\ref{cl.exhaustive} and~\ref{cl.decreasing} implies
 $$\bigcap_{L\in\cF} \De^+_L=\bigcap_{i\in \NN} \De^+_{L_n}.$$
Now $\bigcap_{L\in\cF} \De^+_L=\bigcap_{i\in \NN} \De^+_{L_n}$ is a decreasing sequence of connected compact metric sets, saturated for $\cF$ and therefore is a non-empty connected compact sets saturated for $\cF$. As it does not contain any leaf of $\cF$ one deduces that
$\bigcap_{L\in\cF} \De^+_L\cap \RR^2=\emptyset$ that is 
$\bigcap_{L\in\cF} \De^+_L$ is a compact interval $U$ in $\SS^1_\cF$. 

It remains to show that this interval $U=\bigcap_{L\in\cF} \De^+_L$ is reduced to a point. Otherwise, there is and half  leaf $L_+$ whose limit belongs to the interior  of $U$. According to Claim~\ref{cl.decreasing}, this implies that $L_+\cap \De^+_{L_i}\neq \emptyset$ for every $i$.  This contradics the fact that, for $n$ large enough, the leaf $L_n$ is larger (for $\prec$) than the leaf $L$ carrying the half leaf $L_+$ and thus $L\cap \De^+_{L_n}=\emptyset$. 

This contradiction ends the proof of Proposition~\ref{p.onlyup}. 
\end{proof}

\begin{proof}[Proof of item 1 of Theorem~\ref{t.minimal}]
 Assume that the action of $H$ on $\SS^1_\cF$ is minimal. The foliation $\cF$ cannot be conjugated to the trivial foliation otherwise the set $\{N,S\}$ (unique points in $\SS^1_\cF$ which are not limit of leaves) would be $H$-invariant. 
 
 Thus $\cF$ admits non-separated leaves, and we can assume it is from above (up to change the transverse orientation of $\cF$). 
 If it do not admit non-separated leaves from below then the point $O_\cF$  in $\SS^1_\cF$ given by Proposition~\ref{p.onlyup} is a global fix point of $H$.

\end{proof}

\begin{proof}[Proof of the Item 2 of Theorem~\ref{t.minimal}] We assume that $H$ is a group acting minimally on the leaves of a foliation $\cF$ having non-separated leaves, some of them  from above and some of them from below. According to lemma~\ref{l.orientation} up to consider a finite index subgroup of $H$, acting minimally on the leaves of $\cF$, one may assume that $H$ preserves the orientation and transverse orientation of $\cF$. 
 
Recall that the  ends of regular leaves are dense in $\SS^1_\cF$. Thus it is enough to check that any neighborhood of any end of a regular leaf contains points in the orbit for $H$ of any point of $\SS^1_\cF$. Consider a regular leaf $L$ and $\sigma\colon[-1,1]\to \RR^2$  a segment transverse to $\cF$ with $\sigma(0)\in  L$. We will show that the end of $L_+$ belongs to the closure of any $H$-orbit (the same argument holds for the end of $L^-$). 

We denote by $L_t$ the leaf through $t$. 
Consider the basis of neighborhood  $U^+_t$ of the end $L^+$ given by the compact discs in $\DD^2_\cF$ closure of the half plane  bounded by  $L^+_{-t}$, $\sigma([-t,t])$, and $L^+_{t}$. 

Our hypothesis implies 
\begin{clai}\label{cl.entire} There is a dense subset of values of  $t$ so that $L_t$ is not separated at the right.  As a consequence for every $t$ the topological disc $U^+_t$ contains entire leaves. 
\end{clai}
\begin{proof}The first sentence is directly implied by the existence of leaves which are non-separated at the right, the fact that  $H$ preserves the orientations of the leaves and acts minimaly on the leaves of $\cF$. 

The second sentence have been seen in Section~\ref{s.feuilletage}. 
\end{proof}

Any leaf $L$ cuts $\DD^2_\cF$ in two discs, $\De^+_L$ and $\De^-_L$ (following the tranverse orientation of $\cF$) whose union $\De^+_L\cup\De^-L$ is $\DD^2_\cF$. 
\begin{clai}\label{cl.chaussette} Under the hypotheses, given any $L$ there are $g_1$, $g_2\in H$ so that 
$g_1(\De^+_L)\subset \mathring \De^-_L$ and $g_2(\De^-_L)\subset \mathring\De^+_L$

As a consequence both $\De^+_L$ and $\De^-_L$ contains points in any $H$-orbit of point in $\DD^2_\cF$.
\end{clai}
\begin{proof}We prove the first inclusion, the other is obtained by reversing the transverse orientation of $\cF$. 

Considers $L$ a leaf and $\sigma\colon[-1,1]\to \RR^2$  a segment transverse to $\cF$ (positively oriented for the transverse orientation of $\cF$) so that $\sigma(0)\in L$.  There is  $-t\in[-1,0)$ so that the leaf $L_{-t}$ is non-separated from below from a leaf $L_2$, because the leaves non-separated from below are dense in $\RR^2$, due to the minimality of the action of $H$ on the leaves, and the fact that $H$ preserves the transverse orientation of $\cF$. Thus $L_{-t}\subset \De^-_L$, $L_2\subset  \De^-_L$.  Furthermore $\De^-_{L_2}$ contains $L_{-t}$ and thus contains $L$.  One deduces: 
$$\De^+_{L_2}\subset \De^-_L.$$
Now,  there is 
$h\in H$ so that $h(L_2)=L_{-s}$.with $-s\in(-1,0)$.   In particular one gets that $\De^+_L\subset \mathring \De^+(h(L_2)$ and thus 
$$h^{-1}\De^+_L=\De^+_{h^{-1}(L)}\subset \mathring \De^+(L_2)\subset \mathring\De^-_L.$$
This concludes the proof.
\end{proof}

We are ready to conclude the proof  of Theorem~\ref{t.minimal}: Any neighborhood in $\DD^2_\cF$  of any point of $\SS^1_\cF$ contains an entire leaf $L$ (claim~\ref{cl.entire} above), and thus contains either $\De^+_L$ or $\De^-_L$. According to Claim~\ref{cl.chaussette} this neighborhood contains points in any $H$-orbit of points in $\DD^2_\cF$. This shows the minimality of the action of $H$ on $\SS^1_\cF$, concluding. 
\end{proof}

\begin{theo}\label{t.minimaux}
Let $\cF,\cG$  be two transverse foliations on the plane $\RR^2$. Let $H\subset Homeo(\RR^2)$ be a group preserving both foliations $\cF$ and $\cG$. 
\begin{enumerate}
\item If the action of $H$ on $\SS^1_{\cF,\cG}$ is minimal
 then  both foliations $\cF$ $\cG$ have non-separated leaves
from above and non separated leaves from below.

\item Conversely, if both foliations $\cF$ $\cG$ have non-separated leaves
from above and non separated leaves from below and if
 the orbit of every leaf of $\cF$ and $\cG$ is dense $\RR^2$, then the action of $H$ on $\SS^1_{\cF,\cG}$ is minimal.
 \end{enumerate}
\end{theo}
\begin{proof} For item 1, if the action  of $H$ on $\SS^1_{\cF,\cG}$ is minimal then both actions of $H$ on $\SS^1_{\cF}$ and $\SS^1_{\cG}$ are minimal. Thus item 1 follows from Item 1 of Theorem~\ref{t.minimal}.
 
Conversely, as the action on the leaves of $\cF$ and $\cG$ is assumed to be minimal, and they have non-separated leaves, then Proposition~\ref{p.Hproj} implies that both projections $\Pi_\cF$ and $\Pi_\cG$ are injective. That is $\SS^1_{\cF,\cG}=\SS^1_\cF=\SS^1_\cG$.
Now the minimality of the action of $H$ on this circle at infinity is given by item 2 of Theorem~\ref{t.minimal}.

\end{proof}

\section{Action of the fundamental group on the bifoliated plane of an Anosov flow}

\subsection{The bifoliated plane associated to a Anosov flow}

Let $X$ be an Anosov flow on a closed $3$-manifold $M$. Then Fenley and Barbot show that the lift of $X$ on the universal cover of $M$ is conjugated to $\RR^3,\frac{\partial}{\partial x}$; in particular the space of orbits of this lifted flow is a plane $\cP_X\simeq \RR^2$. Then, the center-stable and center-unstable foliations of $X$ induce (by lifitng on the universal cover and projecting on $\cP_X$) a pair of transverse foliations $\cF^s,\cF^u$ on the plane $\cP_X$.  The triple $(\cP_X,\cF^s,\cF^u)$ is called the \emph{bi-foliated plane} associated to $X$. Finally, the natural action of the fundamental group $\pi_1(M)$ on the universal cover of $M$ projects on $\cP_X$ in an action preserving both foliations $\cF^s$ and $\cF^u$.

Fenley and Barbot proved that, if one of the foliation $\cF^s,\cF^u$ is trivial (that is, has no non-spearated leaf and therefore is conjugate to an affine foliation by parallel straight lines) then the other is also trivial.  In that case, one says that $X$ is \emph{$\RR$-covered}.  In that case  the bifoliated plane is conjugated to one of the two possible models:
\begin{itemize}
 \item the plane $\RR^2$ endowed with the trivial horizontal and vertical foliations; Solodov proved that this is equivalent to the fact that $X$ is orbitally equivalent to the suspension flow of a linear automorphism of the torus $\TT^2$;
 \item the restriction of the trivial horizontal and vertical foliation to the strip $|x-y|<1$.
\end{itemize}

\subsection{Injectivity of the projection of $\DD^2_{\cF^s,\cF^u}$ on $\DD^2_{\cF^s}$ and $\DD^2_{\cF^u}$}

The aim of this section is to prove Theorem~\ref{t.Anosov} which is restated  as Proposition~\ref{p.Anosov} below and Theorem~\ref{t.Anosov-minimal} (in next section).

\begin{prop}\label{p.Anosov} Let $X$ be an Anosov flow on a  $3$-manifold.  Then:
\begin{itemize}
\item Either $X$ is topologically equivalent to the suspension flow of a hyperbolic element of $SL(2,\ZZ)$

\item Or both projections of the compactification $\DD^2_{F^s,F^u}$  on $\DD^2_{F^s}$ and $\DD^2_{F^u}$ are homeomorphisms.
\end{itemize}
\end{prop}
\begin{proof}Assume that the projection on $\DD^2_{\cF^u}$ is not injective. Thus there is a non trivial open interval $I$ of $\SS^1_{\cF^s,\cF^u}$ whose point are not limit of end of leaves of $\cF^u$. Thus a dense subset of point in $I$ are limit of ends of leaves of $\cF^s$. Furthermore, every end of leaf of $\cF^s$ in $I$ is a regular end. Consider a regular end (for instance, a right end) of leaf $L^s_{right}$ of $\cF^s$ whose limit is in the interior of $I$.  Then there is a small unstable segment $\sigma$ through a point of $L^s_{right}$ so that every right half leaf $L^s_{right,t}$  of $\cF^s$ is regular and has its limits in $I$. Then the union of all these  half leaves is what Fenley called a \emph{product region} , in \cite{Fe2}.  Now \cite[Theorem 5.1]{Fe2} asserts that any Anosov  flow  admiting a product region is a suspension flow, concluding.
\end{proof}

\subsection{Minimality of the action on the circle at infinity}
In order to prove Theorem~\ref{t.Anosov} it remais to prove Theorem~\ref{t.Anosov-minimal} below:

\begin{theo}\label{t.Anosov-minimal}
 Let $X$ a flot d'Anosov on a closed  $3$-manifold $M$.  Then $X$ is non-$\RR$-covered if and only if the action of $\pi_1(M)$ on the circle $\SS^1_{F^s,F^u}$ at infinity is minimal.
\end{theo}
\begin{rema} If the manifold $M$ is not orientable and if $X$ is $\RR$-covered, then \cite{Fe1} noticed that $X$ is a suspension flow. Thus, on non-orientable manifolds $M$,  Theorem~\ref{t.Anosov-minimal} asserts the minimality of the action on the circle at infinity, excepted if $M$ is a suspension manifold.
\end{rema}

\begin{rema} The bifoliated plane $(\cP_X,\cF^s,\cF^u)$ remains unchanged if we consider a lift of $X$ on a finite cover.  Thus it is enough to prove Theorem~\ref{t.Anosov-minimal} in the case where $M$ is oriented and the action of $\pi_1(M)$ preserves both orientation and transverse orientation of both foliations $\cF^s$, $\cF^u$.
\end{rema}

Thus, up to now we will assume that $M$ is oriented and the action of $\pi_1(M)$ preserves both orientations and transverse orietations of both foliations $\cF^s$, $\cF^u$.

\begin{rema}If $X$ is $\RR$-covered, then $\SS^1_{\cF^s}$ has exactly $2$ center-like points, which are therefore preserved by the action of $\pi_1(M)$ on $\SS^1_{\cF^s}$: this action is not minimal, and thus the action on $\SS^1_{\cF^s,\cF^u}$ is not minimal.
\end{rema}

Thus we are left to prove Theorem~\ref{t.Anosov-minimal} in the case where $X$ is not $\RR$-covered.We will start with the easier case, when $X$ is assumed to be transitive.  The non-transitive case will be done in the whole next section.

\begin{proof}[Proof of Theorem~\ref{t.Anosov-minimal} when $X$ is transitive]
When $X$ is non-$\RR$-covered and transitive, then \cite{Fe3} proved that $\cF^s$ and $\cF^u$ admits non-separated leaves from above and non-separated leaves from below. As (up to consider a finite cover of $M$), the action of $\pi_1(M)$ preserves the orientation and tranverse orientation of  $\cF^s$, and the action is minimal on the set of leaves of $\cF^s$ thus Theorem~\ref{t.minimal} asserts that the action of $\pi_1(M)$ on $\SS^1_{\cF^s}$ is minimimal.  As $\SS^1_{\cF^s}=\SS^1_{\cF^s,\cF^u}$, this concludes the proof.

\end{proof}

\section{Minimality of the action on the circle at infinity for non-transitive Anosov flows: ending the  proof of Theorem~\ref{t.Anosov}}

For ending the proof of Theorem~\ref{t.Anosov}, we are left to prove :
\begin{theo}\label{t.non-transitive}
Let $X$ be a non-transitive Anosov flow on a closed connected $3$-manifold $M$. Then the action of the fundamental group of $M$ on the circle at infinity is minimal.
\end{theo}

This result is somewhat less intuitive, as the action of the fundamental group $\pi_1(M)$ on the leaves of $M$ is not minimal, and even, if $X$ has several attractors, may fail to admit a leaf whose orbit is dense.

The proof of the minimality of the action on the circle at infinity will require some background on Anosov flow, in particular on non-transitive Anosov flows.
In the whole section, $X$ is a non-transitive Anosov flow on an orientable closed connected manifold $M$ and the natural action of $\pi_1(M)$ on the bifoliated plane $(\cP_X, \cF^s,\cF^u)$ preserves the orientations and transverse orientations of both foliations.  Recall that we have seen that the compactification of both foliations coincide with the one of each foliation.  We will denote by $\DD^2_X, \SS^1_X$ this compactification and the corresponding circle at infinity. We refer by $*$ for this package of hypotheses and notations.

\subsection{Background on non-transitive Anosov flows}

Let $X$ be a non-transitive Anosov flow. Thus, according to \cite{Fe1,Ba1} $X$ is not $\RR$-covered.

The flow $X$ is a structurally stable flow, so that Smale spectral decomposition theorem splits the non-wandering set of $X$ in basic pieces ordered by \emph{Smale order}: a basic piece is upper another if its unstable manifolds cuts the stable manifold of the other. For this order, the maximal basic pieces are the repellers and the minimal are the attractors. In \cite{Br}, Brunella  noticed that the basic pieces are separated by incompressible tori transverse to the flow.

Consider an attractor $\cA$ of $X$. It is a compact set consisting in leaves of the unstable foliation of $X$, hence it is a compact lamination by unstable leaves. Furthermore the intersection of $\cA$ with a transverse segment $\sigma$ is  a Cantor set. An unstable leaf $W^u$ in $\cA$ is called of \emph{boundary type} if $W^u\cap \sigma$ belongs to the boundary of a connected component of $\sigma\setminus \cA$.

A classical results from hyperbolic theory (see for instance \cite{BeBo}) asserts that the unstable leaves in $\cA$ of boundary type are the unstable manifolds of a finite number of periodic orbits called \emph{periodic orbits fo boundary type}. 

The same happens for repellers $\cR$: they are compact laminations by stable leaves, tranversally a Cantor sets, and they admits finitely many boundary leaves, stable manifolds of finitely many periodic orbits called of boundary type.

In this section, we will focus on attractors and repellers.
Consider   an attractor $\cA$  of $X$, its lift $\tilde A$ on the universal cover, and consider the projection of $\cA$ on the bi-foliatioed plane $\cP_X$. This projection is a closed lamination
by leaves of $\cF^s$ and its cuts every tranverse curve along a Cantor set. By a pratical abuse of notation we will still denote by $\cA$ this lamination of $\cP_X$:  thus $\cA$ denotes at the same time a $2$-dimensional lamination on $M$ and a $1$-dimensional lamination on $\cP_X$.   

The same happens for repeller. 

Let $\cA\subset \cP_X$ and $\cR \subset\cP_X$ be the unstable and stable laminations (respectively)  corresponding to an attractor and a repeller of $X$. Then
\begin{itemize}
 \item $\cA\cap \cR=\emptyset$.  This seems obvious, but it will be a crucial property for us: given an unstable leaf $L^u$  and a stable leaf $L^s$, this will be our unique criterion for knowing that they don't intersect. 
 \item the periodic point contained in $\cA$ (reps. $\cR$) are dense in $\cA$ (resp. $\cR$). 
 \item Each periodic orbit of $X$ has a discrete $\pi_1(M)$-orbit in $\cP_X$
 \item the periodic orbits of boundary types are the $\pi_1(M)$-orbits of finitely many $X$-orbits, and therefore are a discrete set in $\cP_X$. 
 \item Fenley \cite{Fe2} shows that the non-separated stable leaves of $\cF^s$ (resp. $\cF^u$) correspond to finitely many orbits of $X$, and hence to a discrete set of periodic points in $\cP_X$
 \item thus the periodic points  $p$ in $\cA$ (resp. $\cR$) which are not of boundary type and whose unstable (resp. stable) leaf
 is regular are dense in $\cA$.
 \item If $\cA_1,\dots, \cA_k$ are the attractors of $X$ then the union of the stable leaves of $\cF^s$ through the laminations $\cA_1,\dots,\cA_K$ of $\cP_X$ are dijoint open seubsets of $\cP_X$ whose union is dense in $\cP_X$.  The same holds for the repellers.

\end{itemize}

As a straightforward consequence one gets: 

\begin{lemm}\label{l.specialpoints}There is a dense subset of $\cP_X$  of points $x$ 
 whose stable leaf $L^s(x)$ contains a periodic point $p$ in an attractor $\cA$, not of boundary type and so that $L^u(p)$ is regular.  A symmetric statement holds for repellers. 
\end{lemm}

\subsection{Proof of Theorem~\ref{t.Anosov-minimal}}
The two main steps of the proof of Theorem~\ref{t.Anosov-minimal} are  Propositions~\ref{p.attractor} and~\ref{p.attracteur}  below. 
\begin{prop}\label{p.attractor}
 Let $L^u$ be a leaf of $\cF^u$ corresponding to an unstable leaf of $X$ contained in a attractor of $X$.
Let $\De_+$ and $\De_-$ be the closures in $\DD^2_X$ of the half planes in $\RR^2$ bounded by $L^s$.  Then there are $g^+,g^-\in \pi_1(M)$ so that $g^-(\De^-)\subset \De^+$ and $g^+(\De^+)\subset \De^-$.

The same statement holds for stable leaves in the repellers.
 \end{prop}
 \begin{coro}\label{c.attractor}
 Let $L^s$ and $L^u$ be leaves of $\cF^s$ and $\cF^u$ in a repeller and in an attractor, respectively.  Let $I\subset\SS^1_X$ be a segment with non empty interior and whose end points are the limit of both ends of the same leaf, $L^s$ or $L^u$.

 Then every orbit of the action of $\pi_1(M)$ contains points in $I$.
 \end{coro}
 \begin{proof} According to Proposition~\ref{p.attractor} there is $g\in \pi_1(M)$ so that $g(\SS^1\setminus I)\subset I$, ending the proof.
 \end{proof}

 \begin{prop}\label{p.attracteur} Given any  non-empty open interval $J\subset \SS^1_X$, there is a $L$ which is either a leaf of $\cF^s$ in a repeller or a leaf of $\cF^u$ in an attractor whose both ends have limits in $J$.

 \end{prop}

\begin{proof}[Proof of Theorem~\ref{t.non-transitive} assuming Propositions~\ref{p.attractor} and~\ref{p.attracteur}]

According to Proposition~\ref{p.attracteur}, every interval $J$ with non-empty interior contains an interval $I$ whose end points are both limit point of the end of a stable or unstable leaf in an a repeller or an attractor, respectively. Now, according to Corollary~\ref{c.attractor}, the interval $I$ contains a point in every $\pi_1(M)$ orbit in $\SS^1_X$.  Thus any $\pi_1(M)$ orbit in $\SS^1_X$ has points in any interval with non-empty interior: in other words, every $\pi_1(M)$ orbit si dense in $\SS^1_X$, or else, the action of $\pi_1(M)$ on $\SS^1_X$ is minimal, ending the proof.
\end{proof}

\subsection{Proof of Proposition~\ref{p.attractor}}

Let $L^u_0$ be an unstable leaf in an attractor $\cA_0$, and $\De^+_0$ be the closure of the upper half plane bounded by $L^u_0$. For proving Proposition~\ref{p.attractor} we want to prove that there is  $f\in\pi_1(M)$ so that $f(\De^-_0)\subset \De^+_0$ (the other announced inclusion is identical).

Consider a point $p_0\in L^u_0$ and $L^s_0$ the stable leaf through $p_0$.

\begin{clai} There is an unstable leaf $L^u_1$ with the following property:
\begin{itemize}\item $L^u_1\subset \De^+_0$
\item $L^u_1$ is contained in the basin of a repeller $\cR_1$
\item $L^u_1$ contains a non-boundary periodic point  $p_1\in L^u_1$  of the repeller $\cR_1$.
 \item  $L^u_1$ cuts the stable leaf $L^s_0$ in a point $L^u_1\cap L^s_0=q_0$.
\end{itemize}
\end{clai}
\begin{proof}The union of unstable leaves in the basin of a repeller and carrying a non-boundary periodic point of this repeller  is dense in $\RR^2$. We can therefore choose such a leaf in $\De^+_0$ and cutting $L^s_0$.
\end{proof}

Let $L^s_1$ be the stable leaf through $p_1$. It is a non-boundary stable leaf contained in the repeller $\cR_1$. Note that $L^s_1$ is disjoint from the attractor $\cA_0$.  Thus
\begin{itemize}\item $L^s_1$ is disjoint from $L^u_0\in\cA_0$.
\item the stable leaf $L^s_1$  is distinct, and therefore disjoint from the stable leaf $L^s_0$.
\end{itemize}
In other words, the union $L^s_0\cup L^u_0$ divides $\cP_X$ in $4$ quadrants and $L^s_1$  contained in one of this quadrants. Let us denote by $C^{\pm,\pm}$ these $4$ quadrants so that
$\De^+_0= C^{-,+}\cup C^{+,+}$ and $L^s_1\subset C^{+,+}$.

Let denote by  $\De^+_1=\De^+(L^s_1)$ the closure of the half plane bounded by $L^s_1$  and contained in $\De^+_0$. Thus $\De^+_1$ is contained in the same quadrant $C^{+,+}$ as $L^s_1$.
We denote by $\De^-_1$ the closure of the other  half plane bounded by $L^s_1$.  Note that $\De^-_1$ contains the $3$ other quadrants, in particular it contains $\De^-_0$ and  $C^{-,+}$.

As the leaf $L^s_1$ is (by assumption) not a boundary leaf of $\cR_1$ it is accumulated on both sides by its $\pi_1(M)$-orbit. Thus there is  a leaf $L^s_2= g(L^s_1)$  in its orbit, cutting  $L^u_1$ at a point $x\in \De^-_1$ arbitrarilly close to $p_1$ and hence $x\in C^{+,+}$.  Notice that $L^s_2$ is contained in the repeller $\cR_1$ and thus is disjoint from $L^u_0\cup L^s_0$. Thus it is contained in one quadrant. As it contained $x\in C^{+,+}$ one has 
$$L^s_2\subset C^{+,+}.$$
 
 Let $h\in\pi_1(M)$ be the generator of the stabilizer $p_1$ so that $L^u_1$ is expanded by $h$.
 We consider the sequence of leaves $h^n(L^s_2)$ which cut $L^u_{1,-}$ at the point $h^n(x)$.
 \begin{clai}For $n$ large enough
 $h^n(L^s_2)$ is contained in the quadrant $C^{-,+}.$
 \end{clai}
\begin{proof}
 
 Each leaf $h^n(L^s_2)$ intersects $L^u_1\subset \De^+_0$ and is disjoint from $L^u_0$ (because $h^n(L^s_2)$ is contained in the repeller).  Hence $h^n(L^s_2)$ is contained in $\De^+_0$, and is distinct and therefore disjoint from $L^s_0$.  Thus $h^n(L^s_2)$ is contained in one of the quadrants $C^{+,+}$ of  $C^{-,+}$.
 
 The point $x_n$ tends to infinity in $L^u_1$ and so goes further $q_0=L^s_0\cap L^u_1$. Thus for $n$ large enough $x_n\in C^{-,+}$.  We proved the for $n$ large enough
 $h^n(L^s_2)\subset C^{-,+}$, proving the claim.
 \end{proof}
 
 We conclude the proof of proposition~\ref{p.attractor} by proving : 
 \begin{clai} Consider $n$ large enough so that $h^n(L^u_2)\subset C^{-,+}$. 
 
 Then either $g(\De^-_1)\subset C^{+,+}\subset \De^+_0$  or $h^ng(\De^-_1)\subset C^{-,+}\subset  \De^+_0$
 \end{clai}
 
 As $\De^-_1$ contains $\De^-_0$ the claim implies that 
 either $g(\De^-_0)\subset \De^+_0$ or $h^ng(\De^-_0)\subset \De^+_0$ which  concludes the proof of Proposition~\ref{p.attractor}. 

 \begin{proof}[Proof of the claim] Assume $g(\De^-_1)$ is not contained in $C^{+,+}$.  As $g(\De^-_1)$ is one of the half plane bounded by $g(L^u_1)= L^u_2\subset C^{+,+}$ one gets that $g(\De^+_1)$ is the half plane bounded by $L^u_2$ and contained in $C^{+,+}$.  In particular,$g(\De^+_1)$ does not contain $q_0$. As $p_1$ and $q_0$ are  on distinct sides of $L^u_2$  one deduces that $p_1\in g(\De^+_1)$. 
 
 As $p_1$ is the fixed point of $h$ one deduces 
 $$p_1\in h^ng(\De^+_1)$$
 Thus $h^ng(\De^+_1)$ is the half plane bounded by $h^n(L^u_2)$ which is not contained in the quadrant $C^{-,+}$.  Thus $h^ng(\De^-_1)$ is the other half plane bounded by $h^n(L^u_2)$ and is contained in $C^{-,+}$, ending the proof.

 \end{proof}

\subsection{Proof of Proposition~\ref{p.attracteur}}
We want to prove that any open interval $I$ in the circle $\SS^1_X$ contains the two ends of an unstable leaf in an attractor or the two ends of a stable leaf of a repeller. 

\begin{lemm}\label{l.regular-periodic} Assuming $*$,  there are dense subsets $E^s_0$, $E^u_0$ of $\SS^1_X$ so that
\begin{itemize}
 \item any $p\in E^s_0$ is the limit of a regular leaf of $\cF^s$ containing a periodic point $x$ which belongs to an attractor $\cA(p)$, and is not of boundary type.
 \item any $q\in E^u_0$ is the limit of a regular leaf of $\cF^u$ containing a periodic point $y$ which belongs to an repeler $\cR(p)$, and is not of boundary type.
\end{itemize}
\end{lemm}
\begin{proof}

According to Lemma~\ref{l.specialpoints} the union of regular stable leaves containing periodic point of non-boundary type of an attractors are dense in $\cP_X$.  This family is therefore separating, according to Lemma~\ref{l.regular}. Thus the limits of their ends is a dense subset of $\SS^1_X$, as announced.

\end{proof}

\begin{lemm}\label{l.regular-repeller} Assuming $*$, there is a dense subset $E\subset \SS^1_X$ so that every $x\in E$ is the limit of the end a regular leaf of $\cF^s$ (resp. $\cF^u$) contained in a repeller $\cR$ (resp. an attractor $\cA$),  and carrying a periodic point of non-boundary type.
\end{lemm}
\begin{proof} Consider a non empty open interval $I\subset \SS^1_X$.  According to Lemma~\ref{l.regular-periodic} there is a point $x\in I$ which is the limit of an end $L^s_+(p_0)$ of a regular leave of $\cF^s$ carrying a periodic point $p_0$ in a non-boundary type unstable leaf $L^u(p_0)$ of a attractor $\cA$.

The point $p_0$ is accumulated on both sides by periodic points in $\cA$. We chose $p_1$ so that the limit $y$ of  $L^s_+(p_1)$ belongs to $I$ (that is possible because $L^s_+(p_0)$ is regular) and $L^s_+(p_1)$ intersects
$L^u(p_0)$ at a point $q_1$. Thus let $J\subset I$ be the segment contained in $I$ and whose end points are $x$ and $y$.Notice that $y\neq x$,that is  $J$ has non-empty interior, as $L^s(p_0)$ is a regular leaf.

Now $L^u(p_0)$ is accumulated on both sides by regular unstable leaves contained in the attractor $\cA$ and containing periodic point of non-boundary type.  Let $L^u_0$ be such a leaf, with non empty intersection with $L^s_+(p_0)$.

If $L^u_0$ does not cut $L^s_+(p_1)$, then one ends is contained in the half strip bounded by $L^s_+(p_0)$, the segment of $[p_0,q_1]^u$ and $L^u_+(q_1)$.  As a consequence, the limit of this end belongs to $I$ and we are done.

Thus we may assume now that $L^u_0$  cuts $L^s_+(p_1)$.

Let $h_0$ and $h_1$ be the generators of the stabilizers of $p_0$ and $p_1$, respectively, so that $h_0$ expands $L^s_+(p_0)$ and $h_1$ expands $L^s_+(p_1)$.

We consider the images $\{h_0^n(L^u_0),h_1^n, n\in\NN(L^u_0)\}$ of the leaf $L^u_0$ by the positive iterates of $h_0$ and $h_1$. Each of these images  is an regular unstable leaf in $\cA$, and has a non-empty intersection with either $L^s_+(p_0)$ or $L^s_+(p_1)$. If one of these leaves does not cross both $L^s_+(p_0)$ and $L^s_+(p_1)$, then it has an end in the segment $J\subset I$, and we are done. 

Assume now that every leaf in $\{h_0^n(L^u_0),h_1^n, n\in\NN(L^u_0)\}$ crosses both $L^s_+(p_0)$ and $L^s_+(p_1)$. These images are leaves of $\cF^u$, and therefore they are either disjoint or equal. For $L\in \{h_0^n(L^u_0),h_1^n, n\in\NN(L^u_0)\}$, let $D(L)\subset \DD^2_\cF$ be the disk obtained as follows:
 one cuts along $L$ the strip bounded by $L^s_+(p_0)$ and $L^s_+(p_1)$, one gets two components;   one considers the closure in $\DD^2_\cF$ of these components; now $D(L)$ is the one  containing  the segment $J\subset \SS^1_\cF$.

The disks $D(L)$ are naturally totally ordered by the inclusion and we fix the indexation $\{h_0^n(L^u_0),h_1^n(L^u_0), n\in\NN\}= \{L^u_n\}$ according to this order: $D(L^u_{n+1}\subset D(L^u_n)$.

Consider $D=\bigcap_n(D(L^u_n)$. It is a compact subset of $\DD^2_\cF$  whose intersection with $\SS^1_\cF$ is the segment $J$.
\begin{clai}
$D\cap  (L^s_+(p_0)\cup L^s_+(p_1))=\emptyset$
\end{clai}
\begin{proof}The leaves $h_0^n(L^u_0)$ have their intersection with $L^s(p_0)$ tending to $x$ as $n\to\infty$: one deduces that $D\cap L^s(p_0)=\emptyset$.  The leaves $h_1^n(L^u_1)$ have their intersection with $L^s(p_1)$ tending to $y$, and thus $D\cap L^s(p_1)=\emptyset$.
\end{proof}

\begin{clai} $D\setminus \SS^1_\cF\neq \emptyset$.

\end{clai}
\begin{proof} there is a point $z$ in the interior of $J$ which is the limit of an end of leaf of $\cF^u$. Thus there is an half unstable leaf $L^u_+$ contained in the strip bounded by $L^s_+(p_0)$ and $L^s_+(p_1)$, and whose limit is $z$. Now $L^u_+$ is disjoint from all the $L^u_n$, and therefore
$$L^u_+\subset D(L^u_n), \forall n$$
This concludes the proof of the claim.
\end{proof}

Consider now a point $t\in  D\setminus \SS^1_\cF$. The leaf $L^u(t)$ is disjoint from the leaves $L^u_n$ for any $n$. Thus it has an empty intersection with $(L^s_+(p_0)\cup L^s_+(p_1))$.  As the consequence one gets
$$L^u(t)\subset D$$
In particular, $L^u(t)$ has both ends on $J$.

Suppose now that the point $t\in  D\setminus \SS^1_X$ as been chosen on the boundary of
$D$.  Thus $t$ is a limit of points in $L^u_n\subset \cA$. As $\cA$ is a closed subset of $\RR^2=\mathring{\DD^2_\cF}$ one deduces that $t\in\cA$, and so $L^u(t)\subset \cA$.

One just found a leaf $L^u(t)$ contained in $\cA$ and having both ends in $J\subset I$.  Let $D_t\subset D$ be the disc bounded by $L^u(t)$.We are not yet done, because $L^u_t$ may fail to be a regular leaf. 

Now Proposition~\ref{p.attractor} implies that every unstable leaf has an image by an element $k\in\pi_1(M)$ which is contained in $D_t$, for instance $L^u_0$.  Now $k(L^u_0)$  is a regular unstable leaf in  an attractor which has the limits of its both ends contained in $J\subset I$, ending the proof.

\end{proof}

We are now ready for ending the proof of Proposition~\ref{p.attracteur}, and therefore of Theorem~\ref{t.non-transitive} which ends the proof of Theorem~\ref{t.Anosov-minimal} and Theorem~\ref{t.Anosov}.

\begin{proof}[Proof of Proposition~\ref{p.attractor}]Let $I\subset \SS^1_X$ be a non-empty open interval.  According to Lemma~\ref{l.regular-repeller} there is a regular unstable leaf $L^u_0$,
contained in an attractor $\cA$ and containing a periodic point of non-boudary type $p_0$, and having an end, say $L^u_{0,+}$, whose limit is a point $x\in I$.

As $L^u_0$ is not a boundary leaf of $\cA$ there are unstable leaves in $\cA$ arbitrarily close to $L^u_0$, on both sides of $L^u_0$. As furthermore $L^u_0$ is a regular leaf, one can chose a leaf $L^u_1\subset \cA$ so  that
\begin{itemize}\item the limit of the end $L^u_{1,+}$ is a point $y\in I$ with $[x,y]\subset I$.
 \item there is a  segment $\sigma$ of a stable leaf having both ends $a$ and $b$ on $L^u_0$ and  $L^u_1$ repsectively.
\end{itemize}
We denote by $D_\sigma$ the disc in $\DD^2_X$ bounded by $\sigma$, $[x,y]$ , $L^u_+(a)\subset L^u_0$ and $L^u_+(b)\subset L^u_1$.

 Now according to Lemma~\ref{l.regular-periodic} there is a point $z\in [x,y]$  which is limit of the end $L^u_+$ of a unstable leaf $L^u$ which carries a periodic point $q$ in a repeller $\cR$, and $q$ is not of boundary type. We denote by $h\in \pi_1(M)$ the generator of the stabilizer of $q$ which is expanding along $L^u$.

 The stable leaf $L^s(q)$ is contained in the repeller $\cR$ and is accumulated on both sides by stable leaves in $\cR$.  We denote by $L^s_0$ a stable leaf in $\cR$ crossing $L^u_+$ at a point $x_0$.

 We consider $L^s_n= h^n(L^s_0)$.  It is a stable leaf in $\cR$ which cuts $L^u_+$ at the point $x_n=h^n(x_0)$.

 Not that $x_n\to z$ as $n\to+\infty$.  In particular, $x_n$ belongs to the disc $D_\sigma$ for $n$ large.

 As $\cR\cap\cA=\emptyset$ the leaves $L^s_n$ are disjoint from $L^u_0$  and  $L^u_1$.  As two distinct stable leaves are disjoint they  are (all but at most one of them) disjoint from $\sigma$.

 So for large $n$, the leaf $L^s_n$ is contained in $D_\sigma$ and therefore as its both ends on $[x,y]\subset I$.

We just exhibit a stable leaf in a repeller, whose both ends are in $I$,  that is we ended the proof of Proposition~\ref{p.attracteur}.
\end{proof}

\vspace{.5cm}

{\bf Christian Bonatti} bonatti@u-bourgogne.fr\\  Institut de Math\'{e}matiques de Bourgogne\footnote{The IMB receives support from the EIPHI Graduate School (contract ANR-17-EURE-0002)}, UMR 5584 du CNRS,  Universit\'{e} de
Bourgogne,\\ 21000, Dijon, France. \vspace{.5cm}

\end{document}